\documentclass[reqno,twoside,11pt]{amsart}
\usepackage{amsmath, amsfonts, amssymb, esint, amsthm, nicefrac, bbm, dsfont}
\usepackage{epsfig, graphicx,xcolor}

\newcommand{\dx}{~\mathrm{d}x}
\newcommand{\sym}{\mathrm{sym}}

\newcommand{\Id}{\mathrm{Id}}

\newcommand{\R}{\mathbb{R}}
\def\endproof{\hspace*{\fill}\mbox{\ \rule{.1in}{.1in}}\medskip }
\newcommand*{\dbar}[1]{\bar{\bar{#1}}}
\newcommand*{\tbar}[1]{\bar{\dbar{#1}}}

\numberwithin{equation}{section}
\theoremstyle{plain}

\newtheorem*{theorem*}{Theorem}
\newtheorem{theorem}{Theorem}[section]
\newtheorem{lemma}[theorem]{Lemma}
\newtheorem{corollary}[theorem]{Corollary}
\newtheorem{proposition}[theorem]{Proposition}

\theoremstyle{definition}

\begin{document}
\title [The Monge-Amp\`ere system in dimension two and codimension three]
{The Monge-Amp\`ere system in dimension two and codimension three}
\author{Dominik Inauen and Marta Lewicka}
\address{D.I.: Institut f\"ur Mathematik, Universit\"at Leipzig,
  D-04109, Leipzig, Germany}
\address{M.L.: University of Pittsburgh, Department of Mathematics, 
139 University Place, Pittsburgh, PA 15260}
\email{dominik.inauen@math.uni-leipzig.de, lewicka@pitt.edu} 

\date{\today}
\thanks{M.L. was partially supported by NSF grant DMS-2407293.
AMS classification: 35J96, 53C42, 53A35}

\begin{abstract}
We revisit the convex integration constructions for the Monge-Amp\`ere system and 
prove its flexibility in dimension $d=2$
and codimension $k=3$, up to $\mathcal{C}^{1,1-1/\sqrt{5}}$. 
To our knowledge, it is the first result in which the obtained
H\"older exponent $1-\frac{1}{\sqrt{5}}$ is larger than $1/2$
but it is not contained in the full flexibility up to
$\mathcal{C}^{1,1}$ result. Previous various approaches, based on Kuiper's
corrugations, always led to the H\"older regularity not exceeding
$\mathcal{C}^{1,1/2}$, while constructions based on the Nash spirals
(when applicable) led to the regularity $\mathcal{C}^{1,1}$. Combining
the two approaches towards an interpolation between their
corresponding exponent ranges has been so far an open problem. 
\end{abstract}

\maketitle

\section{Introduction}

Consider the following Monge-Amp\`ere system, posed on a
$2$-dimensional domain $\omega$, to which we seek a $3$-dimensional
vector field solution $v$:
\begin{equation}\label{MA}
\begin{split}
& \mathfrak{Det}\,\nabla^2 v= f \quad \mbox{ in }\;\omega\subset\R^2,\\
& \mbox{where }\; \mathfrak{Det}\,\nabla^2 v =
\langle \partial_{11}v, \partial_{22}v\rangle -
\big|\partial_{12} v\big|^2\quad \mbox{ for }\; v:\omega\to\R^3.
\end{split}
\end{equation}
System (\ref{MA}) together with its weak formulation (\ref{VK}) below,
was introduced in \cite{lew_conv}, in the full generality of
arbitrary dimension $d$ (now equal $2$) and codimension $k$ (now equal
$3$). Flexibility of both systems, in the sense of
Theorem \ref{th_final} and Corollary \ref{th_weakMA} below, was proved
in there up to the regularity
$\mathcal{C}^{1,\frac{1}{1+d(d+1)/k}}$, and also up to
$\mathcal{C}^{1,1}$ when $k\geq d(d+1)$.
For $d=2$, the former assertion means flexibility up
to $\mathcal{C}^{1,\frac{1}{1+6/k}}$, and when $k=1$ this result agrees
with flexibility up to $\mathcal{C}^{1,\frac{1}{7}}$ obtained in
\cite{lewpak_MA}, which was subsequently improved to
$\mathcal{C}^{1,\frac{1}{5}}$ in \cite{CS} (and to $\mathcal{C}^{1,\frac{1}{1+4/k}}$ for $k$ arbitrary in
\cite{lew_improved}), and further to $\mathcal{C}^{1,\frac{1}{3}}$ in \cite{CHI}.
The findings of \cite{CHI} allowed to obtain \cite{lew_improved2} flexibility up to
$\mathcal{C}^{1,1}$ when $k\geq 4$, 
and up to $\mathcal{C}^{1,\frac{2^k-1}{2^{k+1}-1}}$ for arbitrary $k\geq
1$. Note that this last exponent is less than ${1}/{2}$ for any
$k$ and, in particular, at $k=3$  it equals $7/15$. 

\medskip

\noindent The purpose of this paper is to revisit the previous
constructions and, by adding a new ingredient, 
show flexibility of (\ref{MA}) in dimension $d=2$
and codimension $k=3$ up to: 
$$\mathcal{C}^{1,1-1/\sqrt{5}}.$$ 
To our knowledge,  ours is the first result in which the obtained H\"older exponent is larger than $1/2$
but it is not covered by the full flexibility up to $\mathcal{C}^{1,1}$.
Indeed, the large gap between the exponents' ranges corresponding to
codimensions $3$ and $4$ at the dimension $2$ (or codimensions
$k=d(d+1)-1$ and $k=d(d+1)$ at arbitrary $d$)  
was due to the two different techniques in the
Nash-Kuiper iteration scheme, based, respectively, on Kuiper's corrugations and
on Nash's spirals. Combining the two approaches towards an
interpolation between their resulting exponents has been so far an open problem. 

\medskip

\noindent Recall that the closely related problem of isometric
immersions of a given Riemann metric $g$: 
\begin{equation}\label{II}
\begin{split}
& (\nabla u)^T\nabla u = g \quad\mbox{ in }\;\omega,
\\ &  \mbox{for }\; u:\omega\to \R^{5},
\end{split}
\end{equation} 
reduces to (\ref{MA}) upon taking a family of metrics
$\{g_\epsilon=\Id_2+ 2\epsilon^2 A\}_{\epsilon\to 0}$ each a small perturbation of $\Id_2$ with 
$A:\omega\to\mathbb{R}^{2\times 2}_\sym$ satisfying 
$- \mbox{curl}\,\mbox{curl}\, A = f$.  Making an ansatz $u_\epsilon =
\mbox{id}_2+ \epsilon v +\epsilon^2 w$ and gathering the lowest
order terms in the $\epsilon$-expansions, leads to the following system:
\begin{equation}\label{VK}
\begin{split}
& \frac{1}{2}(\nabla v)^T\nabla v + \sym \nabla w = A\quad\mbox{ in }\;\omega,\\
& \mbox{for }\; v:\omega\to \R^3, \quad w:\omega\to\R^2.
\end{split}
\end{equation}
On a simply connected $\omega$, the system (\ref{VK}) is further
equivalent to: $\mbox{curl}\,\mbox{curl}\, \big(\frac{1}{2}(\nabla
v)^T\nabla v \big) = \mbox{curl}\,\mbox{curl}\, A$, which is
$\mathfrak{Det}\,\nabla^2 v= -\mbox{curl}\,\mbox{curl}\, A$. This
brings us back to  (\ref{MA}), reflecting the agreement of the
Gaussian curvatures $\kappa$ of $g_\epsilon$ 
and of surfaces $u_\epsilon (\omega)$, at their lowest order terms:
\begin{equation*}
\begin{split}
& \kappa (g_\epsilon) = -\epsilon^2\mbox{curl}\,\mbox{curl}\, A + o(\epsilon^2),\\
& \kappa((\nabla u_\epsilon)^T\nabla u_\epsilon) 
= -\frac{\epsilon^2}{2}\mbox{curl}\,\mbox{curl}\, \big((\nabla
v)^T\nabla v + 2\,\sym\nabla w\big)  + o(\epsilon^2) =
{\epsilon^2} \mathfrak{Det}\,\nabla^2 v  + o(\epsilon^2).
\end{split}
\end{equation*}
Our main result pertaining to (\ref{VK}) states that
a $\mathcal{C}^1$-regular pair $(v,w)$  
which is a subsolution, can be uniformly approximated
by exact solutions $\{(v_n,w_n)\}_{n=1}^\infty$, as follows:

\begin{theorem}\label{th_final}
Let $\omega\subset\R^2$ be an open, bounded domain. 
Given the fields $v\in\mathcal{C}^1(\bar\omega,\R^3)$,
$w\in\mathcal{C}^1(\bar\omega,\R^2)$ and 
$A\in\mathcal{C}^{0,\beta}(\bar\omega,\R^{2\times 2}_\sym)$, assume that:
\begin{equation}\label{subsol}
A > \big(\frac{1}{2}(\nabla v)^T\nabla v + \sym\nabla
w\big) \quad \; \mbox{ on } \; \bar\omega,
\end{equation}
in the sense of matrix inequalities. Then, for every
exponent $\alpha$ with:
\begin{equation}\label{VKrange} 
\displaystyle{\alpha<\min\Big\{\frac{\beta}{2},1-\frac{1}{\sqrt{5}}\Big\} }
\end{equation}
and for every $\epsilon>0$, 
there exists $\tilde v\in\mathcal{C}^{1,\alpha}(\bar \omega,\R^k)$,
$\tilde w\in\mathcal{C}^{1,\alpha}(\bar\omega,\R^2)$ such that the following holds:
\begin{align*}
& \|\tilde v - v\|_0\leq \epsilon, \qquad \|\tilde w - w\|_0\leq \epsilon,
\nonumber \vspace{1mm}\\ 
& A -\big(\frac{1}{2}(\nabla \tilde v)^T\nabla \tilde v + \sym\nabla
\tilde w\big) =0 \quad \mbox{ in }\;\bar\omega. \nonumber
\end{align*}
\end{theorem}

\noindent As a byproduct, we obtain the density of 
solutions to (\ref{MA}) in the space of continuous functions: 

\begin{corollary}\label{th_weakMA} 
For any $f\in L^{\infty} (\omega, \R)$ on an open, bounded, simply connected
domain $\omega\subset\mathbb{R}^2$, the following holds.
Fix an exponent $\alpha$ in the range (\ref{VKrange}).
Then, the set of $\mathcal{\mathcal{C}}^{1,\alpha}(\bar\omega, \R^3)$
weak solutions to (\ref{MA}) is dense in $\mathcal{\mathcal{C}}^0(\bar\omega, \R^3)$. 
Namely, every $v\in \mathcal{\mathcal{C}}^0(\bar\omega,\R^3)$ is the
uniform limit of some sequence
$\{v_n\in\mathcal{\mathcal{C}}^{1,\alpha}(\bar\omega,\R^3)\}_{n=1}^\infty$, such that: 
\begin{equation*} 
\mathfrak{Det}\, \nabla^2 v_n  = f \quad \mbox{ on } \; \omega\;
\mbox{ for all }\; n\geq 1.
\end{equation*}
\end{corollary} 

\noindent We now give an overview of the techniques and 
results prior to the present paper.

\smallskip

\subsection{An introduction of the method and the proofs in \cite{lew_conv}.}

Seeking a solution to (\ref{VK}) starts with specifying a subsolution, namely a pair $(v,w)$ that
satisfies (\ref{subsol}). The goal is now to modify the fields $v$,
$w$ within their $\epsilon$-neighbourhoods in
$\mathcal{C}^0(\bar\omega)$, with the goal of canceling the positive definite
defect field $\mathcal{D} = A-
\big(\frac{1}{2}(\nabla v)^T\nabla v + \sym\nabla w\big)$. This is
achieved by an inductive procedure of subsequent modifications,
referred to as the Nash-Kuiper iteration scheme. Therein, each
modification, called a ``step'',  is performed by adding an oscillatory perturbation
with values along a chosen codimension direction $E\in\R^3$,
having the typical form:
\begin{equation}\label{oscil}
x\mapsto \frac{1}{\lambda}a(x)\sin(\lambda\langle x,\eta\rangle) E.
\end{equation} 
Above, $\lambda$ is a (large) amplitude, $sin$ is an appropriate
periodic function, while the effective amplitude $a$ and the
oscillation direction $\eta$
correspond to one of the three modes in the decomposition of
$\mathcal{D}$ into the rank-one ``primitive'' defects:
\begin{equation}\label{dec_base}
\mathcal{D}(x)= \sum_{i=1}^{3}a_i(x)^2\eta_i\otimes \eta_i.
\end{equation}
The number three of these modes is due to the space
$\R^{2\times 2}_{\sym}$ having dimension $3$, and (\ref{dec_base}) is
in fact only valid in the vicinity of a fixed positive definite
matrix, the limitation circumvented by scaling the original $\mathcal{D}$ and adding $\Id_2$
to it. In any case, the aforementioned modification removes the
single mode $a^2\eta\otimes\eta$ of $\mathcal{D}$ at a cost of
introducing a higher order error, consisting of the following three types of terms:
\begin{equation}\label{er_intro}
\frac{a \nabla^2 \langle v, E\rangle }{\lambda} \sin(\lambda\langle x,\eta\rangle), 
\qquad \frac{a \nabla^2 a}{\lambda^2}\sin(\lambda\langle x,\eta\rangle) ,
\qquad \frac{\nabla a\otimes \nabla a}{\lambda^2}\sin(\lambda\langle x,\eta\rangle) .
\end{equation}
Consider the first term above, which is of the leading order. In
magnitude, it is the quotient of the effective amplitude $\|a\|_0$ in
(\ref{oscil}) times the Hessian $\|\nabla^2\langle v, E\rangle \|_0$
of the current field $v$'s appropriate component, 
and the frequency $\lambda$. Since $\|a\|_0$ has the order of
$\|\mathcal{D}\|_0^{1/2}$ as seen from (\ref{dec_base}),  the single
mode $a^2\eta\otimes\eta$ of $\mathcal{D}$ in there
is replaced by the error 
of the order $\|\mathcal{D}\|_0^{1/2}\|\nabla^2 \langle  v,E\rangle \|_0/\lambda$;
while at the same time the second derivatives of $v$ increase by the
factor $\lambda$ due to the form of the perturbation (\ref{oscil}).
The same is true for the second and third modifications, provided that we choose
their codimension directions $E$ linearly independent. Concluding,
after three steps (three modifications), together constituting what is
called a ``stage'', $\|\mathcal{D}\|_0$ decreases by the factor of $\lambda$,
and $\|\nabla^2 v\|_0$ increases by $\lambda$, while the updated field $v$
differs from the original one in its $\mathcal{C}^1$ norm by the order
of $\|a\|_0$ or equivalently of $\|\mathcal{D}\|^{1/2}_0$. 
Continuing in this manner by iterating on stages, we observe
blow-up of the $\mathcal{C}^2$ norm $\|v\|_2$ at the rate $\lambda^{n}$ and the decay of the defect
$\|\mathcal{D}\|_0$ also at the rate $\lambda^n$, which translates to
the control of $\|v\|_1$ at the rate $\lambda^{-n/2}$. By
interpolation, this iterative sequence is Cauchy in
$\mathcal{C}^{1,\alpha}$ and yields a solution to (\ref{VK})  for any exponent $\alpha$ such that $\alpha
-(1-\alpha)/2<0$, namely for $\alpha<1/3$. This is precisely the
result in \cite{lew_conv} for $d=2$, $k=3$.

\medskip

\subsection{The improved method in \cite{CS, lew_improved}.}

The idea employed in \cite{lew_improved} and following
\cite{CS, DIS1/5} was to replace (\ref{dec_base}) by another decomposition:
\begin{equation}\label{dec_better}
\mathcal{D}(x)= a(x)^2\Id_2 + \sym\nabla \Phi(x),
\end{equation}
and transfer its second term in the right hand side into $\sym\nabla w$. Then,
it is only necessary to cancel two rank-one modes $a^2 e_1\otimes e_1$ and $a^2
e_2\otimes e_2$ in $a^2\Id_2$ rather than the three modes in
(\ref{dec_base}), and it is done using two oscillatory
perturbation (\ref{oscil}) achieving values in the two orthogonal codimension directions $E_1$
and $E_2$. The remaining direction $E_3$ may be further used to cancel one
of the modes in the decomposition (\ref{dec_better}) of the resulting
decreased defect. The fourth modification whose frequency we denote by $\mu$, 
removes then the second primitive defect in there, but introduces a further
error, now of the order 
$\|\mathcal{D}\|_0^{1/2}\|\nabla^2 v\|_0\lambda/\mu$. To match this
term with the previous one, we are bound to choose $\mu=\lambda^2$,
which in turn implies the new increase of $\|\nabla^2v\|_0$ by the factor
of $\lambda^2$. After completing six steps of this kind, together constituting
a ``stage'' that removes the defect up to the third order,
we have decreased $\|\mathcal{D}\|_0$ by the factor of
$\lambda^3$, while increasing $\|\nabla^2v\|_0$ by the factor of
$\lambda^2$. Iterating on stages, the
blow-up rate of $\|v\|_2$ is thus $\lambda^{2n}$ and the decay rate of the defect
$\|\mathcal{D}\|_0$ is $\lambda^{3n}$, which translates to
the control of $\|v\|_1$ at the rate $\lambda^{-3n/2}$. Invoking an interpolation
argument as before, the sequence of consecutively perturbed fields $v$
(and likewise $w$) is Cauchy in $\mathcal{C}^{1,\alpha}$ for any exponent $\alpha$ such that $2\alpha
-3(1-\alpha)/2<0$, namely for $\alpha<3/7$. This was the
result obtained in \cite{lew_improved}  for $k=3$.

\medskip

\subsection{A further improvement and proofs in \cite{CHI, lew_improved2}.}

A subsequent improvement of the H\"older regularity exponent as in Theorem
\ref{th_final}, was motivated by the approach of \cite{Kallen} where, before
assigning the effective amplitude $a$ in (\ref{oscil}) and before cancelling
the first mode $a^2 e_1\otimes e_1$, one iterates the decomposition
(\ref{dec_better}) on consecutive errors.  
This way, an entire block of errors, say up to order $N$, may be absorbed at
once. If the single error consisted only of the first type terms in
(\ref{er_intro}), the resulting defect would have the advantageously high order:
\begin{equation}\label{Kallen_err}
\|\mathcal{D}\|_0\left (\frac{\nu}{\lambda}\right )^N,
\end{equation}
where $\nu = \|\nabla^2 v\|_0 / \|\mathcal{D}\|_0^{1/2}$ is the
frequency at which $v$ oscillates. However, the presence of the second
term in (\ref{er_intro}) precludes this 
construction. This is because each $\|\nabla^2a\|_0$ has the order of the second
derivative of the previously
incorporated error which includes terms of the same second type, thus
it is of order $1$ in
$\lambda$. Consequently, the order of the defect does not improve upon iteration.
An ingenious observation put forward in \cite{CHI} allows to 
circumvent this problem, relying on a parallel decomposition
construction of the form (\ref{dec_better}) which, if only applied to $\mathcal{D}$ with
its $\mathcal{D}_{22}$ component removed, trades one
$\partial_1$ derivative for one $\partial_2$ derivative in estimating the
derivatives of $a$.
Hence, the worst $\partial_{11}a$ component of $\nabla^2a$, computed from
the previous error, is now only of the order $1$ in $\nu $ instead of $\lambda$. After $N$ such iterations, the eventual error
quantities are all proportional to (\ref{Kallen_err}). One then proceeds
to canceling the second mode $a^2e_2\otimes e_2$ augmented by the so far neglected
$\mathcal{D}_{22}$ components, through the second modification (a
``step'') of the form (\ref{oscil}), using the same
codimension direction, say $E_1$, as before, a frequency $\mu$ and
the oscillation direction $\eta=e_2$. This construction only affects the component $v_1$ of the
field $v$, increasing its second derivative $\|\nabla^2v_1\|_0$ by a factor $\mu$.
At the same time, the new defect consists of the previously generated
errors bounded by (\ref{Kallen_err}), plus terms of the new order
$\|\mathcal{D}\|_0\|\nabla^2v_1\|_0/(\mu/\lambda)$ due to the neglected
components oscillating on the length scale $\lambda$. This suggests
taking $\mu = \lambda^{N+1}$. Consequently we obtain, to the leading order and when
$N$ is large, the increase of $\|\nabla^2v_1\|_0$ by the factor $\mu$
and the decrease of $\|\mathcal{D}\|_0$ also by the factor $\mu$. If
$k=1$, as in \cite{CHI}, this procedure constitutes a "stage", which
if iterated upon yields the convergence in $\mathcal{C}^{1,\alpha}$,
for $\alpha<1/3$.

\bigskip

\noindent If $k=3$, as in the present case, the procedure in \cite{lew_improved2} (which covers arbitrary $k\geq 1$) continues from the procedure indicated above with the next pair of
modifications (steps),
towards cancellation of the present defect.  The modifications as in (\ref{oscil})
achieve values along the second codimension direction $E_2$ and
oscillate with frequency denoted by $\mu_2$. Only the
component $v_2$ is affected and its second derivatives
$\|\nabla^2v_2\|_0$ are of the order
$\big({\|\mathcal{D}\|_0}/{\mu}\big)^{1/2}\mu_2$, hence we are bound
to take $\mu_2 = \mu^{3/2}$ to achieve matching the order
of $\|\nabla^2v_1\|_0$, namely $\mu$. The current defect decreases by the
factor ${\mu_2/\mu}$, implying the original defect $\|\mathcal{D}\|_0$
decrease by the factor 
$\mu_2$. Finally, the remaining third codimension direction $E_3$ is
employed and we add the two new modifications, now
with a frequency $\mu_3$. The order of the obtained $\|\nabla^2v_3\|_0$ is
$\big({\|\mathcal{D}\|_0}/{\mu_2}\big)^{1/2}\mu_3$, hence we take
$\mu_3 = \mu\mu_2^{1/2} = \mu^{7/4}$ to match with the order $\mu_1$
of $\|\nabla^2v_1\|_0$. The final, third order defect
displays therefore the decrease by the factor ${\mu_3/\mu_2}$, which in terms of 
$\|\mathcal{D}\|_0$ yields the decrease factor 
$\mu_3$. The described construction constitutes a single ``stage''. 
Iterating on stages, the blow-up rate of $\|v\|_2$ becomes $\mu^n$, with
the decrease rate of $\|\mathcal{D}\|_0$ being $\mu^{7n/4}$ and
implying the decay of $\|v\|_1$ at the rate $\mu^{-7n/8}$. By
interpolation, the resulting sequence is Cauchy in $\mathcal{C}^{1,\alpha}$ provided that
$\alpha -7(1-\alpha)/8<0$, i.e. $\alpha<7/15$. This was
precisely the result in \cite{lew_improved2} for $k=3$.

\medskip

\subsection{Proofs in the present paper.} The main new idea
behind the proof of Theorem \ref{th_final} is to iteratively use the 
decomposition (\ref{dec_better}) as described in subsection 1.3., but instead of
cancelling the two principal defects in $a^2\Id_2$ within a single
codimension direction, use now two distinct codimensions. These are assigned to
the defects of consecutive orders in a
circular fashion: $E_1, E_2$ are assigned to the cancellation of the initial defect $\mathcal{D}$
(with $\mu$ as the frequency of the corresponding second modification),
then $E_3, E_1$ to the cancellation of the first order defect (with
frequency $\mu_2$), then
$E_2, E_3$ to the second order defect, then $E_1, E_2$ to the third order, etc.
Upon reaching some prescribed order $K$, we declare this to be a complete
``stage''. The viability of such construction relies on the set
of algebraic conditions satisfied by the progression of frequencies
$\{\mu_i\}_{i=1}^K$. One can check that these indeed hold for a
specific sequence derived from
the Fibonacci sequence, which justifies the presence of the golden
ratio in the resulting H\"older exponent range (\ref{VKrange}). More
precisely, the main ingredient allowing for the flexibility
in Theorem (\ref{th_final}), is the following ``stage'' construction:

\begin{theorem}\label{thm_stage}
Let $\omega\subset\R^2$ be an open, bounded, smooth planar domain. 
Fix two integers $N, K\geq 4$ and an exponent $\gamma\in (0,1)$.
Then, there exists $l_0\in (0,1)$ depending only on $\omega$, and there exists
$\sigma_0\geq 1$ depending on $\omega,\gamma, N, K$,
such that the following holds. Given the fields
$v\in\mathcal{C}^2(\bar\omega+\bar B_{2l}(0),\R^3)$, 
$w\in\mathcal{C}^2(\bar\omega+\bar B_{2l}(0),\R^2)$, 
$A\in\mathcal{C}^{0,\beta}(\bar\omega+\bar B_{2l}(0),\R^{2\times 2}_\sym)$ defined on
the closed $2l$-neighbourhood of $\omega$, and 
given the positive constants $l,
\lambda, \mathcal{M}$ with the properties: 
\begin{equation}\label{Assu}
l\leq l_0,\qquad \lambda^{1-\gamma} l\geq\sigma_0, \qquad
\mathcal{M}\geq\max\{\|v\|_2, \|w\|_2, 1\},
\end{equation}
there exist $\tilde v\in\mathcal{C}^2(\bar \omega+\bar B_{l}(0),\R^k)$,
$\tilde w\in\mathcal{C}^2(\bar\omega+\bar B_{l}(0),\R^2)$ such that, denoting the defects:
\begin{equation}\label{defects}
\mathcal{D}=A -\big(\frac{1}{2}(\nabla v)^T\nabla v + \sym\nabla
w\big), \qquad \tilde{\mathcal{D}} =A -\big(\frac{1}{2}(\nabla \tilde
v)^T\nabla \tilde v + \sym\nabla\tilde w\big), 
\end{equation}
the following bounds are valid:
\begin{align*}
& \begin{array}{l}
\|\tilde v - v\|_1\leq C\lambda^{\gamma/2}\big(\|\mathcal{D}\|_0^{1/2}
+ l\mathcal{M}\big), \vspace{1.5mm}\\ 
\|\tilde w -w\|_1\leq C\lambda^{\gamma}\big(\|\mathcal{D}\|_0^{1/2}
+ l\mathcal{M}\big) \big(1+ \|\mathcal{D}\|_0^{1/2} + l\mathcal{M} +\|\nabla v\|_0\big), 
\end{array}\vspace{3mm} \tag*{(\theequation)$_1$}\refstepcounter{equation} \label{Abound12}\\
&  \begin{array}{l}
\|\nabla^2\tilde v\|_0\leq C \, \displaystyle{\frac{(\lambda
  l)^{(F_{K+1}-2)+(F_{K+1}-1)N/2}\lambda^{\gamma/2}}{l}}\big(\|\mathcal{D}\|_0^{1/2} + l\mathcal{M}\big), 
\vspace{1.5mm}\\ 
\|\nabla^2\tilde w\|_0\leq C \, \displaystyle{\frac{(\lambda
    l)^{(F_{K+1}-2)+(F_{K+1}-1)N/2}\lambda^{\gamma}}{l}} 
\big(\|\mathcal{D}\|_0^{1/2} + l\mathcal{M}\big)\times \vspace{1.5mm}\\
\qquad\qquad\qquad \qquad\qquad\qquad \qquad\qquad\qquad 
\times \big(1+\|\mathcal{D}\|_0^{1/2} + l\mathcal{M}+ \|\nabla v\|_0\big),  
\end{array} \vspace{3mm} \tag*{(\theequation)$_2$}\label{Abound22} \\ 
& \begin{array}{l}
\|\tilde{\mathcal{D}}\|_0\leq C\Big(l^\beta \|A\|_{0,\beta} +
\displaystyle{\frac{\lambda^{\gamma}}{(\lambda
    l)^{2(F_K-1)N}}}\big(\|\mathcal{D}\|_0 +(l\mathcal{M})^2\big)\Big). 
\end{array} \tag*{(\theequation)$_3$} \label{Abound32}
\end{align*}
Above, $\{F_k\}_{k=0}^\infty$ is the Fibonacci sequence, given by the recursion:
$$F_0=F_1=1, \qquad F_{k+2} = F_{k} + F_{k+1} \quad \mbox{for all } \;k\geq 0.$$
The norms of the maps $v, w, A, \mathcal{D}$ and $\tilde v,
\tilde w, \tilde{\mathcal{D}}$ in \ref{Abound12} - \ref{Abound32}
are taken on the respective domains of the maps' definiteness.
The constants $C$ depend only on $\omega, \gamma, N, K$.
\end{theorem}

\smallskip

\noindent 
By assigning $N$ sufficiently large, we see that the quotient $r_{K,N}$ of the
blow-up rate of $\|\nabla^2\tilde v\|_0$
with respect to the rate of decay of $\|\tilde{\mathcal{D}}\|_0$, can
be taken arbitrarily close to $\frac{F_{K+1}-1}{4(F_K-1)}$, whereas
this last quotient approaches the quarter of the golden ratio
$\varphi=\frac{1}{2}(1+\sqrt{5})$ for large $K$:
\begin{equation*}
\begin{split}
\lim_{K\to\infty}\lim_{N\to\infty} r_{K,N} & =
\lim_{K\to\infty}\lim_{N\to\infty} \frac{(F_{K+1}-2) +
  (F_{K+1}-1)N/2}{2N(F_K-1)}\\ & = 
\lim_{N\to\infty}\frac{F_{K+1}-1}{4(F_K-1)} = \frac{1}{4}
\lim_{N\to\infty}\frac{F_{K+1}}{F_K} = \frac{\varphi}{4}. 
\end{split}
\end{equation*}
Since the H\"older regularity exponent deduced from iterating the ``stage'' as
in Theorem \ref{thm_stage} depends only on the aforementioned quotient
of the blow-up rates, and in fact it equals $\frac{1}{1+2r_{K,N}}$ (see
section \ref{sec4} and Theorem \ref{th_NK}), this
implies the range claimed in (\ref{VKrange}).

\smallskip

\subsection{The organization of the paper.} In section \ref{sec_step}
we gather the preparatory results: the convolution and commutator
estimates; the precise formulation of the decomposition
(\ref{dec_better}) and its properties; and the ``step'' construction for the convex
integration algorithm, in which the oscillatory modifications of the
form (\ref{oscil}) are used to cancel a single rank-one primitive
defect of the form $a^2(x)e_i\otimes e_i$. In section \ref{sec_stage}
we carry out the K\"allen-like iterations, shifting the
$\mathcal{D}_{11}$ component of the defect $\mathcal{D}$ onto its
$\mathcal{D}_{22}$ component and absorbing a large
number $N$ of higher order defects obtained by the consecutive application of the
decomposition (\ref{dec_better}). 
Section \ref{sec_stage_prep} determines the sufficient
conditions allowing for the iteration: first the
conditions on the two frequencies $\lambda, \mu$ of any pair of
the modifications that cancel the current
defect; then the conditions on 
the progression of the couples of frequencies $\{\lambda_i,
\mu_i\}_{i=1}^K$ in the $K$ pairs of ``steps'' within a single ``stage''.  In section \ref{sec_fib} we
declare specific frequencies derived from the Fibonacci
sequence, and show that they satisfy the aforementioned
conditions, leading to a proof of Theorem \ref{thm_stage}. Finally, section \ref{sec4}
gathers our previous results on the iteration on stages, which convert the estimates in Theorem
\ref{thm_stage} into the final result of Theorem \ref{th_final}.

\smallskip

\subsection{Notation.}
By $\mathbb{R}^{2\times 2}_{\sym}$ we denote the space of symmetric
$2\times 2$ matrices. The space of H\"older continuous vector fields
$\mathcal{C}^{m,\alpha}(\bar\omega,\R^k)$ consists of restrictions of
all $f\in \mathcal{C}^{m,\alpha}(\mathbb{R}^2,\R^k)$ to the closure of
an open, bounded domain $\omega\subset\R^2$. The
$\mathcal{C}^m(\bar\omega,\R^k)$ norm of such restriction is
denoted by $\|f\|_m$, while its H\"older norm in $\mathcal{C}^{m,
  \gamma}(\bar\omega,\R^k)$ is $\|f\|_{m,\gamma}$. 
By $C$ we denote a universal constant which may change from line to
line, but it depends only on the specified parameters.
For a matrix $D\in\R^{2\times 2}$ we denote by $D^\wedge$ the matrix
with its $D_{22}$ component removed.

\section{Preparatory statements}\label{sec_step}

In this section, we gather the regularization, decomposition and
perturbation statements that will be used in the course of the convex
integration constructions. The first lemma below consists of the basic convolution estimates and
the commutator estimate from \cite{CDS}:

\begin{lemma}\label{lem_stima}
Let $\phi\in\mathcal{C}_c^\infty(\R^d,\mathbb{R})$ be a standard
mollifier that is nonnegative, radially symmetric, supported on the
unit ball $B(0,1)\subset\R^d$ and such that $\int_{\mathbb{R}^d} \phi \dx = 1$. Denote: 
$$\phi_l (x) = \frac{1}{l^d}\phi(\frac{x}{l})\quad\mbox{ for all
}\; l\in (0,1], \;  x\in\R^d.$$
Then, for every $f,g\in\mathcal{C}^0(\mathbb{R}^d,\R)$ and every
$m,n\geq 0$ and $\beta\in (0,1]$ there holds:
\begin{align*}
& \|\nabla^{(m)}(f\ast\phi_l)\|_{0} \leq
\frac{C}{l^m}\|f\|_0,\tag*{(\theequation)$_1$}\vspace{1mm} \refstepcounter{equation} \label{stima1}\\
& \|f - f\ast\phi_l\|_0\leq C \min\big\{l^2\|\nabla^{2}f\|_0,
l\|\nabla f\|_0, {l^\beta}\|f\|_{0,\beta}\big\},\tag*{(\theequation)$_2$} \vspace{1mm} \label{stima2}\\
& \|\nabla^{(m)}\big((fg)\ast\phi_l - (f\ast\phi_l)
(g\ast\phi_l)\big)\|_0\leq {C}{l^{2- m}}\|\nabla f\|_{0}
\|\nabla g\|_{0}, \tag*{(\theequation)$_3$} \label{stima4}
\end{align*}
with a constant $C>0$ depending only on the differentiability exponent $m$.
\end{lemma}

\medskip

\noindent The next auxiliary result, put forward in \cite{CHI} (for a
self-contained proof, see also \cite[Lemma 2.3]{lew_improved2}) is specific to dimension
$d=2$. It allows for the decomposition of the given defect into
a multiple of $\Id_2$ (thus two primitive defects of rank $1$) and a
symmetric gradient, in agreement with the local conformal invariance of
any Riemann 2d metric:

\begin{lemma}\label{lem_diagonal2}
Given a radius $R>0$ and an exponent $\gamma\in (0,1)$, define the linear
space $E$ consisting of $\mathcal{C}^{0,\gamma}$-regular, $\R^{2\times
  2}_\sym$-valued matrix fields $D$ on the ball $\bar B_R\subset \R^2$,
whose traceless part $\dot D = D-\frac{1}{2}(\mbox{trace}\; D) \Id_2$
is compactly supported in $B_R$. There exist linear maps $\bar\Psi$, $\bar a$ in:
\begin{equation*}
\begin{split}
& \bar\Psi: E \to \mathcal{C}^{1,\gamma}(\bar B_R,\R^2), \qquad \bar a: E\to \mathcal{C}^{0,\gamma}(\bar B_R),\\
& E = \big\{D\in \mathcal{C}^{0,\gamma}(\bar B_R,\R^{2\times 2}_\sym); ~ \dot{D}\in
\mathcal{C}^{0,\gamma}_c(B_R,\R^{2\times 2}_\sym)\big\},
\end{split}
\end{equation*} 
with the following properties:
\begin{itemize}
\item[(i)] for all $D\in E$ there holds: $D=\bar a(D)\Id_2 + \sym\nabla \big(\bar\Psi(D)\big),$
\vspace{1mm}
\item[(ii)] $\bar\Psi(\Id_2) \equiv 0$ and $\bar a(\Id_2) \equiv 1$ in $B_R$, \vspace{1mm}
\item[(iii)] $ \|\bar\Psi (D)\|_{1,\gamma}\leq C \|\dot
  D\|_{0,\gamma}$ and $\|\bar a (D)\|_{0,\gamma}\leq C
  \|D\|_{0,\gamma}$ with constants $C$ depending on $R, \gamma$, 
\vspace{1mm}
\item[(iv)] for all $m\geq 1$, if $D\in E\cap \mathcal{C}^{m,\gamma}(\bar B_R,\R^{2\times 2}_\sym)$
  then $\bar\Psi(D)\in \mathcal{C}^{m+1,\gamma}(\bar B_R,\R^2)$ and 
$\bar a(D)\in \mathcal{C}^{m,\gamma}(\bar B_R)$, and we have:
\begin{equation*}
\partial_I\bar\Psi(D)=\bar\Psi(\partial_ID) , \quad \partial_I\bar a(D)=\bar a(\partial_ID)
\qquad \mbox{ for all } \;  |I|\leq m.
\end{equation*}
\item[(v)]  for all $m\geq 1$, if $D\in E\cap
  \mathcal{C}^{m,\gamma}(\bar B_R,\R^{2\times 2}_\sym)$ and if
  additionally $D_{22}=0$ in $\bar B_R$, then:
$$\|\partial_2^{(s)}\bar a(D)\|_{0,\gamma}\leq C \|\partial_2^{(s)}D\|_{0,\gamma}, \quad
\|\partial_1^{(t+1)}\partial_2^{(s)}\bar a(D)\|_{0,\gamma}
\leq C \|\partial_1^{(t)}\partial_2^{(s+1)}D\|_{0,\gamma},$$
for all $s,t\geq 0$ such that $s\leq m$, $t+s+1\leq m$, and with
$C$ depending only on $R,\gamma$.
\end{itemize} 
\end{lemma}

\medskip

\noindent As the final preparatory result, we recall the 
``step'' construction from \cite[Lemma 2.1]{lew_conv}, 
in which a single codimension is used to cancel one rank-one defect of the form
$a(x)^2e_i\otimes e_i$:

\begin{lemma}\label{lem_step2}
Let $v\in \mathcal{C}^2(\R^2, \R^{k})$, $w\in \mathcal{C}^1(\R^2,
\R^{2})$, $\lambda>0$ and $a\in \mathcal{C}^2(\R^2)$ be given. Denote:
$$\Gamma(t) = 2\sin t,\quad \bar\Gamma(t) = \frac{1}{2}\cos (2t),\quad
\dbar\Gamma(t) = -\frac{1}{2}\sin (2t),\quad \tbar\Gamma(t)=1- \frac{1}{2}\cos(2t),$$
and for a fixed  $i=1, 2$ and $j=1\ldots k$ define:
\begin{equation}\label{defi_per2}
\tilde v = v + \frac{a(x)}{\lambda} \Gamma(\lambda x_i) e_j,\quad 
\tilde w = w -\frac{a(x)}{\lambda} \Gamma(\lambda x_i)\nabla v^j
+ \frac{a(x)}{\lambda^2} \bar\Gamma(\lambda x_i)\nabla a(x)
+ \frac{a(x)^2}{\lambda} \dbar\Gamma(\lambda x_i)e_i.
\end{equation}
Then, the following identity is valid on $\R^2$:
\begin{equation}\label{step_err2}
\begin{split}
& \big(\frac{1}{2}(\nabla \tilde v)^T \nabla \tilde v + \sym\nabla \tilde w\big) - 
\big(\frac{1}{2}(\nabla v)^T \nabla v + \sym\nabla w\big) - a(x)^2e_i\otimes e_i
\\ & = -\frac{a}{\lambda} \Gamma(\lambda x_i)\nabla^2 v^j 
+ \frac{a}{\lambda^2} \bar\Gamma(\lambda x_i) \nabla^2 a
+ \frac{1}{\lambda^2}\tbar\Gamma(\lambda x_i)\nabla a\otimes\nabla a.
\end{split}
\end{equation}
\end{lemma}

\section{The K\"allen iteration procedure}\label{sec_stage}

Recall that our Nash-Kuiper iteration scheme will proceed as induction
on ``stages'' specified in Theorem \ref{thm_stage}. Each such
stage is built as an iteration of a given number $K$ of ``steps'',
consecutively decreasing the current defect $\mathcal{D}$ towards the
resulting defect $\tilde{\mathcal{D}}$, given as in (\ref{defects}).  In turn, a single step
consists of a double application of Lemma \ref{lem_step2}.
The first application (corrugation) shifts the $\mathcal{D}_{11}$ component of the current
defect $\mathcal{D}$, after its off-diagonal components have been removed
via an application of Lemma \ref{lem_diagonal2}, onto its $\mathcal{D}_{22}$ component,
which is then removed using the second application of Lemma \ref{lem_step2}.
This defect shift, put forward in \cite{CHI}, is achieved by $N$ iterations of the K\"allen procedure,
presented in this section.

\medskip

\noindent Towards its eventual application in section
\ref{sec_stage_prep}, the field $H$ in Proposition
\ref{prop1} below should be thought of as the (scaled) defect
field $\mathcal{D}$, whereas $Q$ is the (scaled) second derivative $\nabla^2
v$ of the current displacement. For a matrix $D\in\R^{2\times 2}$ we denote by $D^\wedge$ the matrix
with its $D_{22}$ component removed, namely:
$$D^\wedge = D - D_{22}e_2\otimes e_2.$$

\begin{proposition}\label{prop1}
Let $\omega\subset\R^2$ be open, bounded, smooth and let $N\geq 1$ and
$\gamma\in (0,1)$. Then, there exists $l_0\in (0,1)$
depending only on $\omega$ and $\sigma_0\geq 1$ depending on $\omega,
\gamma, N$ such that the following holds. Given the positive constants
$\delta,\eta,\mu,\lambda$ and an integer $M$, satisfying:
\begin{equation}\label{ass_de}
\delta,\eta\leq 2l_0,\qquad \mu\geq \frac{1}{\eta},\qquad
\lambda^{1-\gamma}\geq \mu\sigma_0,\qquad M\geq 0,
\end{equation}
and given the fields $H,Q\in\mathcal{C}^\infty(\bar\omega+\bar
B_{\delta+\eta}(0),\R^{2\times 2}_\sym)$ with the properties:
\begin{equation}\label{ass_HQ0}
\begin{split}
& \|\nabla^{(m)}H\|_0\leq \mu^m \qquad \mbox{for all } \; m=0\ldots M+3N,\\
& \|\nabla^{(m)}Q\|_0\leq \mu^{m+1} \quad \mbox{for all } \; m=0\ldots M+3(N-1),\\
\end{split}
\end{equation}
there exist $a\in \mathcal{C}^\infty(\bar\omega + \bar B_\delta(0),\R)$
and $\Psi\in \mathcal{C}^\infty(\bar\omega + \bar B_\delta(0),\R^2)$
such that, denoting:
$$\displaystyle{\mathcal{F} = a^2\Id_2 + \sym\nabla\Psi - H +
\Big(-\frac{a}{\lambda} \Gamma(\lambda x_1) Q +
\frac{a}{\lambda^2}\bar\Gamma(\lambda x_1)\nabla^2a +
\frac{1}{\lambda^2}\tbar\Gamma(\lambda x_1)\nabla a\otimes\nabla a\Big)^{\wedge}},$$
there hold the estimates:
\begin{align*}
&  \frac{\tilde C}{2}\mu^\gamma\leq a^2\leq \frac{3\tilde C}{2}\mu^\gamma \quad\mbox{ and }\quad
\frac{\tilde C^{1/2}}{2}\mu^{\gamma/2}\leq a\leq \frac{3\tilde C^{1/2}}{2}\mu^{\gamma/2}, 
\tag*{(\theequation)$_1$}\refstepcounter{equation} \label{Ebound11}\\
&  \|\nabla^{(m)} a\|_0\leq C\mu^{\gamma/2}\frac{\lambda^m}{\lambda/\mu}
\quad\mbox{ for all } \; m=1\ldots M, \vspace{3mm} \tag*{(\theequation)$_2$}\label{Ebound21} \\ 
& \|\Psi\|_1\leq C\mu^\gamma \quad \mbox{and}\quad
  \|\nabla^2\Psi\|_0\leq C\mu^\gamma\lambda,  \vspace{3mm} 
\tag*{(\theequation)$_3$} \label{Ebound31}\\
& \displaystyle{\|\nabla^{(m)}\mathcal{F}\|_0 \leq 
C\mu^\gamma\lambda^{\gamma N}\frac{\lambda^m}{(\lambda/\mu)^N}}
\quad\mbox{ for all } \; m=0\ldots M. 
\tag*{(\theequation)$_4$} \label{Ebound41}
\end{align*}
The constant $\tilde C$ in \ref{Ebound11} depends only on $\omega,
\gamma$. Other constants $C$ depend: in \ref{Ebound21} on $\omega,\gamma, N, M$ and only
$M+3(N-1)+1$ derivatives of $H$ and $M+3(N-2)+1$ derivatives of $Q$;
in \ref{Ebound31} on $\omega,\gamma, N$ and only $2+3(N-1)$derivatives
of $H$ and $2+3(N-2)$ derivatives of $Q$; in \ref{Ebound41} on
$\omega, \gamma, N, M$, necessitating the full condition (\ref{ass_HQ0}).
\end{proposition}

\smallskip

\noindent Proposition \ref{prop1} follows from more detailed estimates, that we gather in: 

\begin{theorem}\label{prop1_detail}
Let $\omega, \gamma, N, l_0, \sigma_0$ be as in Proposition
\ref{prop1}. Given the positive parameters
$\delta,\eta,\mu,\lambda, M$ satisfying (\ref{ass_de}),
and given the fields $H,Q\in\mathcal{C}^\infty(\bar\omega+\bar
B_{\delta+\eta}(0),\R^{2\times 2}_\sym)$ such that:
\begin{equation}\label{ass_HQ}
\begin{split}
& \|\nabla^{(m)}H\|_0\leq \mu^m \qquad \mbox{ for all } \; m=0\ldots M+3N,\\
& \|\nabla^{(m)}Q^\wedge\|_0\leq \mu^{m+1} \quad \mbox{for all } \; m=0\ldots M+3(N-1),\\
\end{split}
\end{equation}
there exists a family of fields: 
\begin{equation*}
\big\{a_n\in \mathcal{C}^\infty(\bar\omega + \bar B_\delta(0),\R), \quad
\Psi_n\in \mathcal{C}^\infty(\bar\omega + \bar B_\delta(0),\R^2)\big\}_{n=1}^N
\end{equation*}
such that, denoting $ \mathcal{E}_0 = 0$ and $\displaystyle{\mathcal{E}_n
= -\frac{a_n}{\lambda} \Gamma(\lambda x_1) Q +
\frac{a_n}{\lambda^2}\bar\Gamma(\lambda x_1)\nabla^2a_n +
\frac{1}{\lambda^2}\tbar\Gamma(\lambda x_1)\nabla a_n\otimes\nabla
a_n}$ for $n=1\ldots N$, there holds:
\begin{equation}\label{def_aP}
a_n^2\Id_2 + \sym\nabla \Psi_n = H-\mathcal{E}_{n-1}^\wedge \quad
\mbox{ for all } \; n=1\ldots N,
\end{equation}
together with the following bounds, likewise valid for all $n=1\ldots N$:
\begin{align*}
&  \frac{\tilde C}{2}\mu^\gamma\leq a_n^2\leq \frac{3\tilde C}{2}\mu^\gamma \quad\mbox{ and }\quad
\frac{\tilde C^{1/2}}{2}\mu^{\gamma/2}\leq a_n\leq \frac{3\tilde C^{1/2}}{2}\mu^{\gamma/2}, 
\tag*{(\theequation)$_1$}\refstepcounter{equation} \label{Ebound12}\\
& \|\partial_2^{(s)} a_n^2\|_0\leq C\mu^{\gamma}\mu^s \quad \mbox{ and }
\quad \|\partial_2^s a_n\|_0\leq C\mu^{\gamma/2}\mu^s \qquad
\, \mbox{ for all } \; s=1\ldots M,
\vspace{5mm} \tag*{(\theequation)$_2$}\label{Ebound22} \\ 
&  \hspace{-2.6mm} \left.\begin{array}{l} \|\partial_1^{(t+1)}\partial_2^{(s)} a_n^2\|_0\leq
C \mu^\gamma \displaystyle{\frac{\lambda^{t+1}\mu^{s}}{\lambda/\mu}} \vspace{2mm} \\ 
\qquad \quad \mbox{ and}\quad
\|\partial_1^{(t+1)}\partial_2^{(s)} a_n\|_0\leq C \mu^{\gamma/2}
\displaystyle{\frac{\lambda^{t+1}\mu^{s}}{\lambda/\mu}}\end{array}
\right\} \qquad \mbox{ for all }\; s+t+1=1\ldots M,
\vspace{3mm} \tag*{(\theequation)$_3$}\label{Ebound32} \\ 
& \|\Psi_n\|_1\leq C\mu^\gamma \quad \mbox{and}\quad
  \|\nabla^2\Psi_n\|_0\leq C\mu^\gamma\lambda,  \vspace{3mm} 
\tag*{(\theequation)$_4$} \label{Ebound42}\\
& \|\partial_1^{(t)}\partial_2^{(s)} \big(\mathcal{E}_n^\wedge -\mathcal{E}_{n-1}^\wedge\big)\|_0 \leq
C \mu^\gamma \lambda^{\gamma(n-1)}\frac{\lambda^t\mu^s}{(\lambda/\mu)^n}
\quad \qquad \; \;\,
\mbox{ for all }\; s+t=0\ldots M.
\tag*{(\theequation)$_5$} \label{Ebound52}
\end{align*}
The constant $\tilde C$ in \ref{Ebound12} depends only on $\omega,
\gamma$. Other constants $C$ depend: in \ref{Ebound22}, \ref{Ebound32}
on $\omega,\gamma, n, M$ and only
$M+3(n-1)+1$ derivatives of $H$ and $M+3(n-2)+1$ derivatives of $Q^\wedge$;
in \ref{Ebound42} on $\omega,\gamma, n$ and only $2+3(n-1)$derivatives
of $H$ and $2+3(n-2)$ derivatives of $Q^\wedge$; in \ref{Ebound52} on
$\omega, \gamma, N, M$, necessitating the full condition (\ref{ass_HQ}).
\end{theorem}

\begin{proof}
{\bf 1. (Preparatory observations)} We fix a radius $R>0$ depending
only on $\omega$, such that $\bar\omega + \bar B_4(0)\subset
B_R(0)$. Also, given $0<\delta, \eta\leq 2l_0$ we set a cut-off function $\chi\in
\mathcal{C}^\infty_c(\omega+B_{\delta+\eta}(0), [0,1])$ with $\chi\equiv 1$ on
$\bar\omega+\bar B_{\delta}(0)$.
When $l_0\ll 1$, it is pos\-sible to request that for any function
$f\in\mathcal{C}^m(\bar\omega+\bar B_{\delta+\eta}(0),\R)$ and any multiindex $I$
with $|I|\leq m$, there holds:
\begin{equation}\label{chi_st}
\|\partial_I(\chi f)\|_0\leq C \sum_{I_1+ I_2=I}
\|\partial_{I_1}\chi\|_0 \|\partial_{I_2}f\|_0\leq C\sum_{I_1+ I_2=I}
\mu^{|I_1|}\|\partial_{I_2}f\|_0,
\end{equation}
with constants $C$ depending only on $\omega$ and $m$.
Further, we directly observe that the second bound in \ref{Ebound12} is 
implied by the first one, because:
$$\big|\frac{a(x)}{\tilde C^{1/2}\mu^{\gamma/2}}-1\big|\leq
\big|\frac{a(x)^2}{\tilde C\mu^{\gamma}}-1\big|\leq \frac{1}{2}.$$
Next, in both \ref{Ebound22},
\ref{Ebound32} the second bounds follow from the first ones, in view of 
\ref{Ebound12}, and they necessitate the same number of derivatives
bounds in condition (\ref{ass_HQ}). 
In \ref{Ebound22}, we use the one-dimensional Fa\'a di Bruno formula:
\begin{equation*}
\begin{split}
\|\partial_2^{(s)}a_n\|_0 & \leq C\Big\|\sum_{p_1+2p_2+\ldots
  sp_s=s} a_n^{2(1/2-p_1-\ldots -p_s)}\prod_{z=1}^s\big|\partial_2^{(z)}a_n^2\big|^{p_z}\Big\|_0
\\ & \leq C\|a_n\|_0 \sum_{p_1+2p_2+\ldots
  sp_s=s}\prod_{z=1}^s\Big(\frac{\|\partial_2^{(z)}a_n^2\|_0}{\tilde
  C \mu^\gamma}\Big)^{p_z}\leq C\mu^{\gamma/2}\mu^s.
\end{split}
\end{equation*}
For \ref{Ebound32}, we apply the multivariate version of the Fa\'a di
Bruno formula. Let $\Pi$ be the set of all partitions $\pi$ of the initial
multiindex $\{1\}^{t+1}+\{2\}^s$ into multiindices $I$ of lengths
$|I|\in [1, t+s+1]$. Denoting
by $|\pi|$ the number of multiindices in a partition $\pi$, we have:
\begin{equation*}
\begin{split}
\|\partial_1^{(t+1)}\partial_2^{(s)}a_n\|_0 & \leq C\Big\|\sum_{\pi\in
  \Pi} a_n^{2(1/2-|\pi|)}\prod_{I\in\pi}\partial_I a_n^2\Big\|_0
 \leq C\|a_n\|_0
 \sum_{\pi\in\Pi}\prod_{I\in\pi}\frac{\|\partial_Ia_n^2\|_0}{\tilde C \mu^\gamma}
\\ & \leq C \mu^{\gamma/2}\lambda^{t+1}\mu^s\sum_{\pi\in\Pi} 
\Big(\prod_{I\in\pi,\;1\in I}\frac{1}{\lambda/\mu}\Big)
\leq C\mu^{\gamma/2} \frac{\lambda^{t+1}\mu^s}{\lambda/\mu}.
\end{split}
\end{equation*}
Finally, applying Fa\'a di Bruno's formula to 
inverse rather than square root, \ref{Ebound22}, \ref{Ebound32} yield:
\begin{equation}\label{asm2}
\begin{split}
&\|\partial_2^{(s)}\Big(\frac{1}{a_{n+1}+a_{n}}\Big)\|_0  \leq
 \frac{C}{\mu^{\gamma/2}}\sum_{p_1+2p_2+\ldots
  sp_s=s}\prod_{z=1}^s\Big(\frac{\|\partial_2^{(z)}(a_{n+1}+a_{n})\|_0}{\tilde
  C^{1/2}\mu^{\gamma/2}}\Big)^{p_z}\leq \frac{C}{\mu^{\gamma/2}}\mu^s, \\
&\|\partial_1^{(t+1)}\partial_2^{(s)}\Big(\frac{1}{a_{n+1}+a_{n}}\Big)\|_0
 \leq C \Big\|\sum_{\pi\in
  \Pi} \big(a_{n+1}+a_{n})^{-1-|\pi|}\prod_{I\in\pi}\partial_I(a_{n+1}+a_{n})\Big\|_0
\\ &  \qquad\qquad \qquad\qquad\qquad \quad \;
\leq \frac{C}{\mu^{\gamma/2}}
 \sum_{\pi\in\Pi}\prod_{I\in\pi}\frac{\|\partial_I a_{n+1}\|_0 +\|\partial_Ia_n\|_0}
{\tilde C^{1/2}\mu^{\gamma/2}}
\\ & \qquad\qquad \qquad\qquad\qquad \quad \;
\leq \frac{C}{\mu^{\gamma/2}}\lambda^{t+1}\mu^s 
\sum_{\pi\in\Pi} \Big(\prod_{I\in\pi,\;1\in I}\frac{1}{\lambda/\mu}\Big) 
\leq \frac{C}{\mu^{\gamma/2}}\frac{\lambda^{t+1}\mu^s}{\lambda/\mu}.
\end{split}
\end{equation}
with the the same dependence of constants and using the same number of
derivatives of $H$ and $Q^\wedge$ as in the corresponding bounds on
$a_n$ and $a_{n+1}$.

\smallskip

{\bf 2. (Induction base $n=1$ and definition of $\tilde C$)}
Let the linear maps ${\bar a, \bar\Psi}$ be as in 
Lemma \ref{lem_diagonal2}, applied with the specified $R$, $\gamma$.
From (\ref{chi_st}) and using the bound on $\|H\|_1$ in (\ref{ass_HQ}), we get:
$$\|\bar a(\chi H)\|_0\leq C\|\chi H\|_{0,\gamma}\leq C\big(\|H\|_0 +
\|H\|_0^{1-\gamma} \|\nabla(\chi H)\|_0^\gamma \big)\leq C\mu^\gamma$$
where $C$ depends on $\omega$, $\gamma$. We declare $\tilde C$ to
be four times the final constant above, leading to:
\begin{equation}\label{Ct_st}
\|\bar a(\chi H)\|_0\leq \frac{\tilde C}{4}\mu^\gamma.
\end{equation}
This results in the validity of the first bound in \ref{Ebound12} in view of
(\ref{Ct_st}), where we set:
\begin{equation}\label{def_aP_global1}
a_1^2 = \tilde C\mu^\gamma + \bar a(\chi H), \qquad
\Psi_1=\bar\Psi(\chi H) - \tilde C\mu^\gamma id_2,
\end{equation}
while the identity (\ref{def_aP}) holds (on $\bar\omega + \bar B_\delta(0)$ where
$\chi\equiv 1$), because:
\begin{equation*}
\begin{split}
\chi H & = \bar a(\chi H)\Id_2 + \sym \nabla \big(\bar\Psi(\chi H)\big) 
\\ & = (a_1^2\Id_2 - \tilde C\mu^\gamma \Id_2 ) + (\sym\nabla \Psi_1 + \tilde
C\mu^\gamma \Id_2) = a_1^2\Id_2 + \sym\nabla \Psi_1.
\end{split}
\end{equation*}
Further, using (\ref{chi_st})  and the bound on $\|H\|_{M+1}$ in
(\ref{ass_HQ}), we obtain for all $m=1\ldots M$:
\begin{equation}\label{a1_st}
\begin{split}
\|\nabla^{(m)}a_1^2\|_0& = \|\bar a(\nabla^{(m)} (\chi H))\|_0
\leq C \|\nabla^{(m)}(\chi H)\|_{0,\gamma} \\ & \leq C \big(
\|\nabla^{(m)}(\chi H)\|_0 +  \|\nabla^{(m)}(\chi H)\|_0^{1-\gamma}
\|\nabla^{(m+1)}(\chi H)\|_0^\gamma\big) \\ & \leq C\big(\mu^m +
\mu^{m(1-\gamma)}\mu^{(m+1)\gamma}\big) \leq C \mu^\gamma\mu^m,
\end{split}
\end{equation}
where $C$ depends on $\omega, \gamma, M$. The above implies both
bounds \ref{Ebound22} and \ref{Ebound32} for $a_1^2$. Towards \ref{Ebound52}, we apply
(\ref{a1_st}) and use (\ref{ass_HQ}) to estimate $\|H\|_{M+3}$ and
$\|Q^\wedge\|_{M}$ in:
\begin{equation*}
\begin{split}
& \|\partial_1^{(t)}\partial_2^{(s)}\mathcal{E}_1^\wedge\|_0\leq  
C\sum_{\tiny\begin{array}{c} p_1+q_1+z_1=t\\ q_2+z_2=s\end{array}}
\Big( \lambda^{p_1-1}\|\nabla^{(q_1+q_2)}a_1\|_0\|\nabla^{(z_1+z_2)}Q^\wedge\|_0 
\\ & \quad + 
\lambda^{p_1-2}\|\nabla^{(q_1+q_2+1)} a_1\|_0\|\nabla^{(z_1+z_2+1)}a_1\|_0 + 
\lambda^{p_1-2}\|\nabla^{(q_1+q_2)} a_1\|_0\|\nabla^{(z_1+z_2+2)}a_1\|_0 \Big)
\\ & \leq C\sum_{\tiny\begin{array}{c} p_1+q_1+z_1=t\\  q_2+z_2=s\end{array}}
\big(\lambda^{p_1-1}\mu^{\gamma/2}\mu^{q_1+q_2}\mu^{z_1+z_2+1}
+ \lambda^{p_1-2}\mu^\gamma\mu^{q_1+q_2+z_1+z_2+2}\big)
\\ & \leq  C \mu^\gamma\Big(\frac{\lambda^t\mu^s}{\lambda/\mu}
+ \frac{\lambda^t\mu^s}{(\lambda/\mu)^2}\Big)
\leq C \mu^\gamma \frac{\lambda^t\mu^s}{\lambda/\mu},
\end{split}
\end{equation*}
valid for $t+s\leq M$, with $C$ depending on $\omega, \gamma, M$.
The above is precisely \ref{Ebound52} since $\mathcal{E}_0=0$.

\smallskip

{\bf 3. (Induction step: bounds \ref{Ebound12} --
\ref{Ebound32})} Assume that \ref{Ebound12} --
\ref{Ebound32} and \ref{Ebound52}  hold up to some 
$1\leq n\leq N-1$, necessitating (\ref{ass_HQ}) to estimate
derivatives of $H$ only up to ${M+3n}$ and of $Q^\wedge$ up to ${M+3(n-1)}$.
We will prove the validity of \ref{Ebound12} --
\ref{Ebound32} and \ref{Ebound52} at $n+1$, necessitating (\ref{ass_HQ}) to estimate
$\|H\|_{M+3n+1}$ and $\|Q^\wedge \|_{M+3(n-1)+1}$. We start by
noting that, as a consequence of
\ref{Ebound52} in view of (\ref{chi_st}), for all $j=1\ldots n$ and
$s+t=0\ldots M$ we have:
\begin{equation*}
\|\partial_1^{(t)}\partial_2^{(s)}\big(\chi (\mathcal{E}_j- \mathcal{E}_{j-1})^\wedge\big)\|_{0}
\leq C\frac{\mu^\gamma}{\lambda^\gamma}\lambda^t\mu^s
\Big(\frac{\lambda^{\gamma}}{\lambda/\mu}\Big)^j,
\end{equation*}
with $C$ depending on $\omega,\gamma, n, M$. This yields: 
\begin{equation}\label{imtih1}
\|\partial_1^{(t)}\partial_2^{(s)}\big(\chi (\mathcal{E}_j- \mathcal{E}_{j-1})^\wedge\big)\|_{0,\gamma}
\leq C\mu^\gamma\lambda^t\mu^s
\Big(\frac{\lambda^{\gamma}}{\lambda/\mu}\Big)^j,
\end{equation}
necessitating the bounds on $\|H\|_{M+3n+1}$ and $\|Q^\wedge \|_{M+1+3(n-1)+1}$.
Since $\mathcal{E}_0=0$, it follows that:
\begin{equation}\label{imtih2}
\|\partial_1^{(t)}\partial_2^{(s)}(\chi \mathcal{E}_n)\|_{0,\gamma}
\leq C\mu^\gamma \lambda^t\mu^s
\sum_{j=1}^n \Big(\frac{\lambda^{\gamma}}{\lambda/\mu}\Big)^j
\leq C \mu^\gamma{\lambda^t\mu^s}
\frac{\lambda^\gamma/(\lambda/\mu)}{1- \lambda^\gamma/(\lambda/\mu)} \leq  C \mu^\gamma
\lambda^\gamma \frac{\lambda^t\mu^s}{\lambda/\mu}  
\end{equation}
as $\lambda^\gamma/(\lambda/\mu)\leq 1/\sigma_0\leq 1/2$ from the third
assumption in (\ref{ass_de}). In particular, we get:
$$\|\bar a(\chi\mathcal{E}_n^\wedge)\|_0\leq
C\|\chi\mathcal{E}_n^\wedge\|_{0,\gamma}\leq C\mu^\gamma 
\frac{\lambda^\gamma}{\lambda/\mu} \leq \frac{\tilde C}{4}\mu^\gamma $$
provided that $\sigma_0$ is large enough, in function of $\omega, \gamma, n$.
Recalling (\ref{Ct_st}), the above yields the well definiteness of
$a_{n+1}^2$ together with the first bound in \ref{Ebound12}, upon defining:
\begin{equation}\label{def_aP_global}
a_{n+1}^2 = \tilde C\mu^\gamma + \bar a(\chi H -
\chi\mathcal{E}_n^\wedge),\qquad 
\Psi_{n+1} = \bar\Psi(\chi H - \chi\mathcal{E}_n^\wedge) - \tilde
C\mu^\gamma id_2.
\end{equation}
The identity (\ref{def_aP}) clearly holds from Lemma
\ref{lem_diagonal2} (i).
Towards proving \ref{Ebound22}, we apply (\ref{imtih2}) and the bound on
$\|\nabla^{(m)}(\chi H)\|_{0,\gamma}$ in (\ref{a1_st}), to get:
$$\| \partial_2^{(s)}a_{n+1}^2\|_0\leq C \|\partial_2^{(s)}(\chi H -
\chi\mathcal{E}_n^\wedge)\|_{0,\gamma} \leq C\mu^\gamma\mu^s\Big(1+ 
\frac{\lambda^\gamma}{\lambda/\mu}\Big) \leq C\mu^\gamma\mu^s$$
for all $s=1\ldots M$. The estimate in \ref{Ebound32} follows
by additionally recalling Lemma \ref{lem_diagonal2} (v) in:
\begin{equation*}
\begin{split}
\| \partial_1^{(t+1)}\partial_2^{(s)}a_{n+1}^2\|_0
& \leq C \|\partial_1^{(t+1)}\partial_2^{(s)}(\chi H)\|_{0,\gamma} +
C\|\partial_1^{(t)}\partial_2^{(s+1)}(\chi\mathcal{E}_n^\wedge)\|_{0,\gamma}
\\ & \leq C\mu^\gamma\Big(\mu^{s+t+1} + \lambda^t\mu^{s+1}
\frac{\lambda^\gamma}{\lambda/\mu}\Big)  \leq C\mu^\gamma
\lambda^t\mu^{s+1}\leq 
C\mu^\gamma\frac{\lambda^{t+1}\mu^s}{\lambda/\mu},
\end{split}
\end{equation*}
for all $s+t+1=1\ldots M$. In both bounds above the constant $C$ depends on $\omega,\gamma, n,
M$ and condition (\ref{ass_HQ}) has been used 
only up to ${M+3n}+1$ in derivatives of $H$ and up to ${M+3(n-1)}+1$ in $Q^\wedge$.

\smallskip

{\bf 4. (Induction step: the bound \ref{Ebound52})} In this step we
continue the inductive step argument and show \ref{Ebound52} at $n+1$,
necessitating (\ref{ass_HQ}) to estimate 
derivatives of $H$ up to ${M+3(n+1)}$ and of $Q^\wedge$ up to
${M+3n}$. From the definitions (\ref{def_aP_global1}) and
(\ref{def_aP_global}), it follows that: 
$$a_{n+1}^2 - a_n^2 = \bar a(\chi(\mathcal{E}_n - \mathcal{E}_{n-1})^\wedge).$$
Consequently, recalling (\ref{imtih1}) we get:
\begin{equation*}
\begin{split}
& \|\partial_2^{(s)} (a_{n+1}^2 - a_n^2)\|_0\leq C\|\partial_2^{(s)}(\chi(\mathcal{E}_n
 - \mathcal{E}_{n-1})^\wedge)\|_{0,\gamma}\leq C\mu^\gamma \mu^s 
\big(\frac{\lambda^\gamma}{\lambda/\mu}\big)^n \\
&\|\partial_1^{(t+1)}\partial_2^{(s)} (a_{n+1}^2 - a_n^2)\|_0\leq
C\|\partial_1^{(t)} \partial_2^{(s+1)}(\chi(\mathcal{E}_n
 - \mathcal{E}_{n-1})^\wedge)\|_{0,\gamma}\leq C\mu^\gamma \lambda^{t}\mu^{s+1} 
\big(\frac{\lambda^\gamma}{\lambda/\mu}\big)^n,
\end{split}
\end{equation*}
which together with (\ref{asm2}) implies:
\begin{equation}\label{imtih3}
\begin{split}
& \|\partial_2^{(s)} (a_{n+1} - a_n)\|_0\leq C
\sum_{p+q=s} \|\partial_2^{(p)}(a_{n+1}^2 - a_n^2)\|_0
\big\|\partial_2^{(q)}\big(\frac{1}{a_{n+1}+a_n}\big)\big\|_0 \\ &
\qquad \qquad \qquad \qquad \leq 
C\mu^{\gamma/2} \mu^s\big(\frac{\lambda^{\gamma}}{\lambda/\mu}\big)^n \\
&\|\partial_1^{(t+1)}\partial_2^{(s)} (a_{n+1} - a_n)\|_0\\ & \qquad
\qquad \leq
C \hspace{-6mm} \sum_{\tiny\begin{array}{c} p_1+q_1=t+1\\p_2+q_2=s\end{array}} 
\hspace{-4mm}\|\partial_1^{(p_1)}\partial_2^{(p_2)}(a_{n+1}^2 - a_n^2)\|_0
\big\|\partial_1^{(q_1)}\partial_2^{(q_2)}\big(\frac{1}{a_{n+1}+a_n}\big)\big\|_0
\\ & \qquad \qquad 
\leq  C\mu^{\gamma/2} \frac{\lambda^{t+1}\mu^s}{\lambda/\mu} 
\big(\frac{\lambda^{\gamma}}{\lambda/\mu}\big)^n
\end{split}
\end{equation}
for all $s=0\ldots M$ and $s+t+1=1\ldots M$, with $C$ depending on
$\omega, \gamma, n, M$ and where we used bounds in (\ref{ass_HQ}) on
the derivatives of $H$ up to $M+3n+1$ and of $Q^\wedge$ up to $M+3(n-1)+1$.
Towards proving \ref{Ebound52}, we first write:
\begin{equation*}
\begin{split}
\mathcal{E}_{n+1}^\wedge-\mathcal{E}_n^\wedge = & -\frac{a_{n+1}-a_{n}}{\lambda}
\Gamma(\lambda x_1)Q^\wedge +  \frac{1}{\lambda^2} \bar\Gamma(\lambda x_1)
\big(a_{n+1}\nabla^2a_{n+1}-a_n\nabla^2a_n\big)^\wedge \\ & + 
\frac{1}{\lambda^2} \tbar\Gamma(\lambda x_1)\big(\nabla a_{n+1}\otimes \nabla a_{n+1}- 
\nabla a_{n}\otimes \nabla a_{n}\big)^\wedge
\end{split}
\end{equation*}
with the goal of estimating, for all $s+t=0\ldots M$, contribution of
the three terms above in the quantity
$\|\partial_1^{(t)} \partial_2^{(s)}\big(\mathcal{E}_{n+1}-\mathcal{E}_n\big)^\wedge\|_0$. 
The first term is:
\begin{equation*}
\begin{split}
& \big\|\partial_1^{(t)}\partial_2^{(s)}\Big(\frac{a_{n+1}-a_{n}}{\lambda}
\Gamma(\lambda x_1)Q^\wedge\Big)\big\|_0 
\leq  C \hspace{-5mm}\sum_{\tiny\begin{array}{c} p_1+q_1+z_1=t\\q_2+z_2=s\end{array}}
\hspace{-5mm}\lambda^{p_1-1}\|\partial_1^{(q_1)}\partial_2^{(q_2)}(a_{n+1} - a_n)\|_0
\big\|\nabla^{(z_1+z_2)}Q^\wedge\big\|_0 \\ & \leq
C\mu^{\gamma/2}\big(\frac{\lambda^\gamma}{\lambda/\mu}\big)^n 
\hspace{-5mm}\sum_{\tiny\begin{array}{c}
    p_1+q_1+z_1=t\\q_2+z_2=s\end{array}}  \hspace{-5mm}\lambda^{p_1-1}
\lambda^{q_1}\mu^{q_2}\mu^{z_1+z_2+1} 
\leq C\mu^{\gamma/2}\big(\frac{\lambda^\gamma}{\lambda/\mu}\big)^n   
\frac{\lambda^{t}\mu^s}{\lambda/\mu}
= C\mu^{\gamma/2} \lambda^{\gamma n}\frac{\lambda^t\mu^s}{(\lambda/\mu)^{n+1}},
\end{split}
\end{equation*}
where we used (\ref{imtih3}) and the bounds in (\ref{ass_HQ}) on
the derivatives of $H$ up to $M+3n+1$ and of $Q^\wedge$ up to $M+3(n-1)+1$.
For the second term, we employ the identity:
$a_{n+1}\nabla^2a_{n+1}-a_n\nabla^2a_n
= (a_{n+1}-a_n) \nabla^2a_{n} +a_{n+1}\nabla^2(a_{n+1}- a_n)$ and thus obtain:
\begin{equation*}
\begin{split}
& \big\|\partial_1^{(t)}\partial_2^{(s)}\Big(\frac{1}{\lambda^2} \bar\Gamma(\lambda x_1)
\big(a_{n+1}\nabla^2a_{n+1}-a_n\nabla^2a_n\big)^\wedge\Big) \big\|_0 
\\  & \leq C \hspace{-7mm}\sum_{\tiny\begin{array}{c} p_1+q_1+z_1=t\\q_2+z_2=s\end{array}}
\hspace{-7mm}\lambda^{p_1-2}\Big( \|\partial_1^{(q_1)}\partial_2^{(q_2)}(a_{n+1}-a_n)\|_0 
\big(\|\partial_{1}^{(z_1+2)}\partial_2^{(z_2)}a_{n}\|_0
+ \|\partial_{1}^{(z_1+1)}\partial_2^{(z_2+1)}a_{n}\|_0\big)
\\ &  \qquad \qquad\qquad +  \big(\|\partial_1^{(q_1+2)}\partial_2^{(q_2)}(a_{n+1}-a_n)\|_0  
+ \|\partial_1^{(q_1+1)}\partial_2^{(q_2+1)}(a_{n+1}-a_n)\|_0 \big)
\|\partial_{1}^{(z_1)}\partial_2^{(z_2)}a_{n+1}\|_0\Big) \\ & 
\leq C \mu^\gamma \big(\frac{\lambda^\gamma}{\lambda/\mu}\big)^n 
\hspace{-5mm}\sum_{\tiny\begin{array}{c} p_1+q_1+z_1=t\\q_2+z_2=s\end{array}}
\hspace{-5mm}\lambda^{p_1-2}\Big(\lambda^{q_1}\mu^{q_2}\frac{\lambda^{z_1+2}\mu^{z_2}}{\lambda/\mu}
+ \frac{\lambda^{q_1+2}\mu^{q_2}}{\lambda/\mu}
\lambda^{z_1}\mu^{z_2}\Big) \\ & 
\leq C\mu^{\gamma}\big(\frac{\lambda^\gamma}{\lambda/\mu}\big)^n   
\frac{\lambda^{t}\mu^s}{\lambda/\mu}
= C\mu^{\gamma} \lambda^{\gamma n}\frac{\lambda^t\mu^s}{(\lambda/\mu)^{n+1}},
\end{split}
\end{equation*}
where we used (\ref{imtih3}), the already established bounds
\ref{Ebound22}, \ref{Ebound32} up to counter $n+1$, and the bounds in (\ref{ass_HQ}) on
the derivatives of $H$ up to counter $M+3n+3$ and of $Q^\wedge$ up to
$M+3(n-1)+3$. For the third term, we use the identity: $\nabla a_{n+1}\otimes \nabla a_{n+1}-  
\nabla a_{n}\otimes \nabla a_{n} = \nabla
(a_{n+1}-a_n)\otimes \nabla (a_{n+1}-a_n) + 2\,\sym \big(  
\nabla (a_{n+1}-a_n)\otimes \nabla a_{n}\big)$ and estimate:
\begin{equation*}
\begin{split}
& \big\|\partial_1^{(t)}\partial_2^{(s)}\Big(\frac{1}{\lambda^2}
\tbar\Gamma(\lambda x_1)\big(\nabla a_{n+1}\otimes \nabla a_{n+1}-  
\nabla a_{n}\otimes \nabla a_{n}\big)^\wedge\Big) \big\|_0 
\\  & \leq C \hspace{-7mm}\sum_{\tiny\begin{array}{c} p_1+q_1+z_1=t\\q_2+z_2=s\end{array}}
\hspace{-7mm}\lambda^{p_1-2}\Big(\|\partial_1^{(q_1+1)}\partial_2^{(q_2)}(a_{n+1}-a_n)\|_0 
\big(\|\partial_{1}^{(z_1)}\partial_2^{(z_2)}\nabla (a_{n+1}-a_n)\|_0
+ \|\partial_{1}^{(z_1)}\partial_2^{(z_2)}\nabla a_{n}\|_0 \big)
\\ &  \qquad \qquad\qquad \qquad + \|\partial_1^{(q_1)}\partial_2^{(q_2+1)}(a_{n+1}-a_n)\|_0  
\|\partial_{1}^{(z_1+1)}\partial_2^{(z_2)}a_{n}\|_0\Big) \\ & 
\leq C \mu^\gamma \big(\frac{\lambda^\gamma}{\lambda/\mu}\big)^n 
\hspace{-5mm}\sum_{\tiny\begin{array}{c} p_1+q_1+z_1=t\\q_2+z_2=s\end{array}}
\hspace{-5mm}\lambda^{p_1-2}\Big(
\frac{\lambda^{q_1+1}\mu^{q_2}}{\lambda/\mu} \frac{\lambda^{z_1+1}\mu^{z_2}}{\lambda/\mu}
+ \lambda^{q_1}\mu^{q_2+1}
\frac{\lambda^{z_1+1}\mu^{z_2}}{\lambda/\mu} \Big) \\ & 
\leq C\mu^{\gamma}\big(\frac{\lambda^\gamma}{\lambda/\mu}\big)^n   
\frac{\lambda^{t}\mu^s}{(\lambda/\mu)^2}
= C\mu^{\gamma} \lambda^{\gamma n}\frac{\lambda^t\mu^s}{(\lambda/\mu)^{n+2}},
\end{split}
\end{equation*}
where we used (\ref{imtih3}), the already established bounds
\ref{Ebound22}, \ref{Ebound32}, and the bounds in (\ref{ass_HQ}) on
the derivatives of $H$ up to $M+3n+2$ and of $Q^\wedge$ up to
$M+3(n-1)+2$.  In conclusion:
\begin{equation*}
\begin{split}
\|\partial_1^{(t)}&\partial_2^{(s)}\big(\mathcal{E}_{n+1}-\mathcal{E}_n\big)^\wedge\|_0 
\leq C\mu^{\gamma} \lambda^{\gamma n}\frac{\lambda^t\mu^s}{(\lambda/\mu)^{n+1}},
\end{split}
\end{equation*}
which is exactly \ref{Ebound52} at $n+1$ and with the right dependence
of constants and order of used derivatives, as claimed.

\smallskip

{\bf 5. (The bound \ref{Ebound42})} From the definitions
(\ref{def_aP_global1}), (\ref{def_aP_global}) we obtain, for all
$n=1\ldots N$ in virtue of Lemma \ref{lem_diagonal2} (iii),
(\ref{imtih2}) and the third assumption in (\ref{ass_de}): 
\begin{equation*}
\begin{split}
& \|\Psi_n\|_1\leq C\big(\mu^\gamma + \|\chi H\|_{0,\gamma} + 
\|\chi \mathcal{E}_{n-1}^\wedge\|_{0,\gamma}\big) \leq C\mu^\gamma \big (1
+ \frac{\lambda^\gamma}{\lambda/\mu}\big)\leq C \mu^\gamma
\\ & \|\nabla^2\Psi_n\|_0\leq C\big(\|\nabla(\chi H)\|_{0,\gamma} + 
\|\nabla(\chi \mathcal{E}_{n-1}^\wedge)\|_{0,\gamma}\big) \leq C\mu^\gamma \big (\mu
+ \lambda \frac{\lambda^\gamma}{\lambda/\mu}\big)\leq C \mu^\gamma\lambda,
\end{split}
\end{equation*}
with $C$ depending on $\omega,\gamma, n$ and where we used the bounds in (\ref{ass_HQ}) on
the derivatives of $H$ up to $2+3(n-1)$ and of $Q^\wedge$ up to
$2+3(n-2)$. This ends the proof of Theorem \ref{prop1_detail}.
\end{proof}

\bigskip

We observe that  Proposition \ref{prop1} easily follows from Theorem
\ref{prop1_detail}:

\medskip

\noindent {\bf Proof of Proposition \ref{prop1}}

Declare $a=a_N$, $\Psi = \Psi_N$. Then,
\ref{Ebound12} becomes \ref{Ebound11}, while \ref{Ebound22} and
\ref{Ebound32} imply \ref{Ebound21} since $\mu\leq\lambda$. Further,
\ref{Ebound42} becomes \ref{Ebound31}, and \ref{Ebound41} follows from
\ref{Ebound52} and (\ref{def_aP}), because:
$$\mathcal{F} = a_N^2\Id_2 + \sym\nabla\Psi_N - H +
\mathcal{E}_N^\wedge = \mathcal{E}_N^\wedge - \mathcal{E}_{N-1}^\wedge.$$
The proof is done.
\endproof 

\section{The double step construction and its iteration}\label{sec_stage_prep}

The proof of Theorem \ref{thm_stage} relies on
iterating the corrugation construction which utilizes
Proposition \ref{prop1} in one codimension direction $e_\alpha$ and
augments it by cancelling the bulk of the defect $\mathcal{D}_{22}$, now accumulated in
its $\mathcal{D}_{22}$ component, via the
application of Lemma \ref{lem_step2} in another codimension direction $e_\beta$.
In preparation for this recursion, carried out in Proposition
\ref{prop3}, we first present its building block, relative to
a chosen pair $\alpha\neq \beta$, and whose bounds 
we index using the eventual recursion counter $k$. The given quantities are
referred to through the subscript $k$ while the derived quantities
carry the consecutive subscript $k+1$. Namely, we have:

\begin{proposition}\label{prop2}
Let $\omega\subset\R^2$ be open, bounded, smooth and let $N\geq 1$ and
$\gamma\in (0,1)$. Then, there exists $l_0\in (0,1)$
depending only on $\omega$, and $\sigma_0\geq 1$ depending on $\omega,
\gamma, N$ such that the following holds. Given the positive constants
$\delta,\eta$, the positive frequencies:
$$\mu_{k-1}\leq\lambda_k\leq\mu_k\leq\lambda_{k+1}\leq\mu_{k+1},$$ 
the positive auxiliary constants $\tilde C_k, A_k, B_k$, and an integer $M$, satisfying:
\begin{equation}\label{ass_de2}
\begin{split}
& \delta,\eta\leq 2l_0,\qquad \mu_k\geq \frac{1}{\eta},\qquad
\lambda_{k+1}^{1-\gamma}\geq \mu_k\sigma_0,\qquad M\geq 0,\\
& \frac{\mu_k}{\mu_{k-1}}\geq \big(\frac{A_k}{\tilde C_k}\big)^{1/2}, \qquad
\frac{\mu_{k+1}}{\lambda_{k+1}}\geq
\max\Big\{\big(\frac{\lambda_{k+1}}{\mu_k}\big)^{(N-1)/2}, 
\big(\frac{\lambda_{k+1}}{\mu_k}\big)^{N-1} \frac{(B_k/\tilde C_k)^{1/2}}{\mu_k/\lambda_k}\Big\},
\end{split}
\end{equation}
and given  $v_{k}\in\mathcal{C}^\infty(\bar\omega+\bar
B_{\delta+\eta}(0),\R^{3})$, $w_{k}\in\mathcal{C}^\infty(\bar\omega+\bar
B_{\delta+\eta}(0),\R^{2})$, $A_0\in\mathcal{C}^\infty(\bar\omega+\bar
B_{\delta+\eta}(0),\R^{2\times 2}_\sym)$, such that together with the derived field $\mathcal{D}_{k}
= A_0 - \big(\frac{1}{2} (\nabla v_k)^T\nabla v_k +\sym\nabla w_k\big)$, the
following bounds hold relative to the two chosen components $\alpha\neq \beta\in \{1,2,3\}$:
\begin{equation}\label{ass_HQ2}
\begin{split}
& \|\nabla^{(m)}\mathcal{D}_k\|_0\leq \tilde C_k\mu_k^m \qquad 
\qquad \quad \mbox{for all } \; m=0\ldots M+3(N+1),\\
& \|\nabla^{(m+2)}v^\alpha_k\|_0\leq A_k^{1/2}\mu_{k-1}^{m+1}
\qquad \; \, \mbox{for all } \; m=0\ldots M+3N,\\ 
& \|\nabla^{(m+2)}v^\beta_k\|_0\leq B_k^{1/2}\lambda_k^{m+1} 
\qquad \; \,\mbox{for all } \; m=0\ldots M+1,
\end{split}
\end{equation}
there exist $v_{k+1}\in\mathcal{C}^\infty(\bar\omega+\bar
B_{\delta}(0),\R^{3})$ and $w_{k+1}\in\mathcal{C}^\infty(\bar\omega+\bar
B_{\delta}(0),\R^{2})$, such that denoting the new derived field ${\mathcal{D}}_{k+1}
= A_0 - \big(\frac{1}{2} (\nabla v_{k+1})^T\nabla v_{k+1} +\sym\nabla w_{k+1}\big)$,
there hold the estimates:
\begin{align*}
& \begin{array}{l}  \|(v_{k+1}- v_k)^\alpha\|_1+ \|(v_{k+1}-
  v_k)^\beta\|_1 \leq C \tilde C_k^{1/2}\mu_k^{\gamma/2},\end{array}\vspace{4mm} 
\tag*{(\theequation)$_1$}\refstepcounter{equation} \label{P2bound1}\\
& \begin{array}{l} \displaystyle{\|\nabla^{(m+2)}v_{k+1}^\alpha\|_0\leq C\tilde
  C_k^{1/2}\mu_k^{\gamma/2} \lambda_{k+1}^{m+1} } \vspace{2mm}\\
\displaystyle{\|\nabla^{(m+2)}v_{k+1}^\beta\|_0\leq C\tilde
  C_k^{1/2}\mu_k^{\gamma/2} \mu_{k+1}^{m+1}}
\end{array}
\qquad \qquad
\mbox{ for all } \; m=0\ldots M,\vspace{3mm} \tag*{(\theequation)$_3$} \label{P2bound2}\\
&  \begin{array}{l} \displaystyle{\| w_{k+1}-w_k\|_1\leq C \tilde
C_k^{1/2}\mu_k^{\gamma/2}\big(\|\nabla v_k^\alpha\|_0 +\|\nabla
v_k^\beta\|_0 + \tilde C_k^{1/2}\mu_k^{\gamma/2}\big),} \vspace{2mm}\\
\displaystyle{\|\nabla^2( w_{k+1}-w_k)\|_0\leq C \tilde
C_k^{1/2}\mu_k^{\gamma/2}\big(\|\nabla v_k^\alpha\|_0 +\|\nabla
v_k^\beta\|_0 + \tilde C_k^{1/2}\mu_k^{\gamma/2}\big) \mu_{k+1},}
\end{array}
\vspace{4mm} \tag*{(\theequation)$_2$}\label{P2bound3} \\ 
& \begin{array}{l} \displaystyle{\|\nabla^{(m)}{\mathcal{D}}_{k+1}\|_0\leq C\tilde
C_k\mu_k^{\gamma}\lambda_{k+1}^{\gamma N}\frac{\mu_{k+1}^m}{(\lambda_{k+1}/\mu_k)^N}}
\quad \; \; \mbox{ for all } \; m=0\ldots M. \end{array}
\tag*{(\theequation)$_4$} \label{P2bound4}
\end{align*}
Above, constants $C$ depend: in \ref{P2bound1}, \ref{P2bound3} on $\omega,\gamma, N$;
in \ref{P2bound2}, \ref{P2bound4} on $\omega, \gamma, N, M$.
\end{proposition}

\begin{proof}
{\bf 1. (Applying Proposition \ref{prop1})} We apply Proposition
\ref{prop1} to the same $\omega,\gamma,N,l_0$, $\sigma_0$, $\delta,\eta$,
$M, N$ as in there, and with:
$$\mu=\mu_k,\quad \lambda=\lambda_{k+1},\quad H=\frac{1}{\tilde C_k}\mathcal{D}_k,
\quad Q=\frac{1}{\tilde C_k^{1/2}}\nabla^2 v_k^\alpha,$$
upon validating conditions (\ref{ass_HQ0}) in view of (\ref{ass_HQ2}) and (\ref{ass_de2}):
\begin{equation*}
\begin{split}
& \|\nabla^{(m)} H\|_0\leq \mu_k^m \qquad\qquad\qquad \qquad\qquad
\qquad \quad \, \mbox{for all }\; m=0\ldots M+3(N+1),\\
& \|\nabla^{(m)} Q\|_0\leq \big(\frac{A_k}{\tilde
  C_k}\big)^{1/2}\frac{\mu_{k-1}}{\mu_k}\cdot \mu_k^{m+1} \leq \mu_k^{m+1}\qquad\mbox{for
  all }\; m=0\ldots M+3N.
\end{split}
\end{equation*}
Having thus obtained the fields $a, \Psi, \mathcal{F}$ with properties
\ref{Ebound11} -- \ref{Ebound41}, we define $a_{k+1}\in \mathcal{C}^\infty(\bar\omega +
\bar B_\delta(0), \R)$, $\Psi_{k+1}\in\mathcal{C}^\infty(\bar\omega +
\bar B_\delta(0),\R^2)$, $\mathcal{F}_{k+1}\in
\mathcal{C}^\infty(\bar\omega+\bar B_\delta(0),\R^{2\times 2}_\sym)$ by: 
\begin{equation*}
a_{k+1} = \tilde C_{k}^{1/2}a, \qquad \Psi_{k+1}=\tilde C_k\Psi,
\qquad \mathcal{F}_{k+1} = \tilde C_k\mathcal{F}, 
\end{equation*}
so that:
\begin{equation}\label{def_Fk+1}
\begin{split}
& \mathcal{F}_{k+1} =  a_{k+1}^2\Id_2 + \sym\nabla \Psi_{k+1} - \mathcal{D}_k 
\\ & ~ +\Big(-\frac{a_{k+1}}{\lambda_{k+1}}
\Gamma (\lambda_{k+1}x_1)\nabla^2v_k^\alpha +
\frac{a_{k+1}}{\lambda_{k+1}^2}\bar\Gamma(\lambda_{k+1}x_1)\nabla^2a_{k+1}
+ \frac{1}{\lambda_{k+1}^2}\tbar\Gamma(\lambda_{k+1}x_1)(\nabla
a_{k+1})^{\otimes 2}\Big)^\wedge.
 \end{split}
\end{equation}
and the following bounds hold:
\begin{align*}
&  \frac{\tilde C\tilde C_k}{2}\mu_k^\gamma\leq a_{k+1}^2\leq
\frac{3\tilde C \tilde C_k}{2}\mu_k^\gamma \quad\mbox{ and }\quad
\frac{\tilde C^{1/2} \tilde C_k^{1/2}}{2}\mu_k^{\gamma/2}\leq
a_{k+1}\leq \frac{3\tilde C^{1/2} \tilde C_k^{1/2}}{2}\mu_k^{\gamma/2}, 
\tag*{(\theequation)$_1$}\refstepcounter{equation} \label{Ebound13}\\
&  \|\nabla^{(m)} a_{k+1}\|_0\leq C\tilde
C_k^{1/2}\mu_k^{\gamma/2}\frac{\lambda_{k+1}^m}{\lambda_{k+1}/\mu_k} 
\quad\mbox{ for all } \; m=1\ldots M+5, \vspace{3mm} \tag*{(\theequation)$_2$}\label{Ebound23} \\ 
& \|\Psi_{k+1}\|_1\leq C\tilde C_k\mu_k^\gamma \quad \mbox{and}\quad
  \|\nabla^2\Psi_{k+1}\|_0\leq C\tilde C_k\mu_k^\gamma\lambda_{k+1},  \vspace{3mm} 
\tag*{(\theequation)$_3$} \label{Ebound33}\\
& \displaystyle{\|\nabla^{(m)}\mathcal{F}_{k+1}\|_0 \leq 
C \tilde C_k \mu_k^\gamma\lambda_{k+1}^{\gamma N}\frac{\lambda_{k+1}^m}{(\lambda_{k+1}/\mu_k)^N}}
\quad\mbox{ for all } \; m=0\ldots M. 
\tag*{(\theequation)$_4$} \label{Ebound43}
\end{align*}
where the constant $\tilde C$ depends only on $\omega, \gamma$, while
the constants $C$ depend on: in \ref{Ebound33} on $\omega,\gamma,N$;
in \ref{Ebound23}, \ref{Ebound43} on $\omega,\gamma, N,M$.

\smallskip

{\bf 2. (Adding the first corrugation)} We define the
intermediate fields $\tilde v \in\mathcal{C}^\infty(\bar\omega +\bar B_{\delta}(0),\R^3)$,
$ \tilde w\in\mathcal{C}^\infty(\bar\omega +\bar B_{\delta}(0), \R^2)$ by setting,
in accordance with Lemma \ref{lem_step2} for $i=1,j=\alpha$:
\begin{equation*}
\begin{split}
& \tilde v  = v_{k} + \frac{a_{k+1}}{\lambda_{k+1}} \Gamma(\lambda_{k+1} x_1)e_\alpha,\\
& \tilde w = w_{k} - \frac{a_{k+1}}{\lambda_{k+1}} \Gamma(\lambda_{k+1} x_1)\nabla v^\alpha_k 
+ \frac{a_{k+1}}{\lambda_{k+1}^2}\bar\Gamma(\lambda_{k+1}x_1)\nabla a_{k+1} 
+ \frac{a_{k+1}^2}{\lambda_{k+1}}\dbar\Gamma(\lambda_{k+1} x_1)e_1 + \Psi_{k+1}.
\end{split}
\end{equation*}
By \ref{Ebound13}, \ref{Ebound23}, we directly get for all $m=0\ldots M$:
\begin{equation}\label{jamea0}
\begin{split}
& \|(\tilde v - v_{k})^\alpha\|_1 \leq C \big(\|a_{k+1}\|_0 +
\frac{\|\nabla a_{k+1}\|_0}{\lambda_{k+1}}\big)  \leq C\tilde C_k^{1/2}\mu_k^{\gamma/2}\\ 
& \|\nabla^{(m+2)}(\tilde v - v_{k})^\alpha\|_0 \leq C \sum_{p+q=m+2}\lambda_{k+1}^{p-1}
\|\nabla^{(q)}a_{k+1}\|_0 \leq C\tilde C_k^{1/2}\mu_k^{\gamma/2}\lambda_{k+1}^{m+1},
\end{split}
\end{equation}
where $C$ in the first bound depends on $\omega,\gamma, N$, while in
the second bound on $\omega,\gamma,N, M$. Combining that last bound
with (\ref{ass_de2}) yields:
\begin{equation}\label{jamea1}
\|\nabla^{(m+2)}\tilde v^\alpha\|_0 \leq C \tilde C_k^{1/2}\mu_k^{\gamma/2}\lambda_{k+1}^{m+1} +
A_k^{1/2}\mu_{k-1}^{m+1} \leq C\tilde C_k^{1/2}\mu_k^{\gamma/2}\lambda_{k+1}^{m+1}.
\end{equation}
Similarly, \ref{Ebound13} -- \ref{Ebound33} and (\ref{ass_HQ2}) result in the estimates:
\begin{equation}\label{jamea2}
\begin{split}
& \|\tilde w - w_{k}\|_1\leq C\Big(\tilde C_{k}\mu_k^\gamma
 + \|a_{k+1}\|_0\|\nabla v_k^\alpha\|_0 + \|a_{k+1}\|_0^2 
\\ & \qquad\qquad \qquad \quad  +\frac{\|\nabla a_{k+1}\|_0\|\nabla
  v_k^\alpha\|_0 + \|a_{k+1}\|_0\|\nabla^2 v_{k}^\alpha\|_0 +
  \|a_{k+1}\|_0 \|\nabla a_{k+1}\|_0}{\lambda_{k+1}} \\
& \qquad\qquad \qquad \quad  +\frac{\|\nabla a_{k+1}\|_0\|\nabla^2
  a_{k+1}\|_0 + \|\nabla a_{k+1}\|_0^2}{\lambda_{k+1}^2} \Big)\\
& \qquad \qquad \; \; \leq C \tilde C_k^{1/2}\mu_k^{\gamma/2} 
\big(\|\nabla v_k^\alpha\|_0 + \tilde C_k^{1/2}\mu_k^{\gamma/2}\big),\\
& \|\nabla^2 (\tilde w - \bar  w_{k})\|_0\leq 
C \Big( \tilde C_k\mu_k^\gamma \lambda_{k+1}+ 
\lambda_{k+1}\big(\|a_{k+1}\|_0\|\nabla v_k^\alpha\|_0+ \|a_{k+1}\|_0^2\big) 
\\ & \qquad\qquad\qquad \qquad \quad
+ \big(\|\nabla a_{k+1}\|_0\|\nabla v_k^\alpha\|_0 + \|a_{k+1}\|_0\|\nabla^2 v_k^\alpha\|_0  
+ \|a_{k+1}\|_0\|\nabla a_{k+1}\|_0\big)
\\ & \qquad\qquad\qquad \qquad\quad + \frac{\|\nabla^2a_{k+1}\|_0\|\nabla v_k^\alpha\|_0 + \|\nabla
  a_{k+1}\|_0\|\nabla^2 v_k^\alpha\|_0 + \|a_{k+1}\|_0\|\nabla^3 v_k^\alpha\|_0}{\lambda_{k+1}} 
\\ & \qquad\qquad\qquad\qquad\quad + \frac{\|a_{k+1}\|_0\|\nabla^2a_{k+1}\|_0  + \|\nabla
  a_{k+1}\|_0^2}{\lambda_{k+1}}
\\ & \qquad\qquad\qquad \qquad\quad 
+ \frac{\|a_{k+1}\|_0 \|\nabla^3a_{k+1}\|_0 + \|\nabla
  a_{k+1}\|_0\|\nabla^2 a_{k+1}\|_0}{\lambda_{k+1}^2}\Big) 
\\ & 
 \qquad\qquad \qquad \; \; \leq C \tilde C_k^{1/2}\mu_k^{\gamma/2} 
\big(\|\nabla v_k^\alpha\|_0 + \tilde
C_k^{1/2}\mu_k^{\gamma/2}\big)\lambda_{k+1,}
\end{split}
\end{equation}
with constants $C$ depending on $\omega,\gamma, N$.
Finally, we note that (\ref{step_err2}) and (\ref{def_Fk+1}) yield:
\begin{equation}\label{jamea3}
\begin{split}
& \tilde{\mathcal{D}} \doteq A_0 - \big(\frac{1}{2}(\nabla \tilde v)^T\nabla \tilde  v +
\sym\nabla \tilde w\big) \\ &  = \mathcal{D}_{k} - a_{k+1}^2e_1\otimes
e_1 - \sym\nabla\Psi_{k+1}
\\ & \quad - \Big(-\frac{a_{k+1}}{\lambda_{k+1}} \Gamma(\lambda_{k+1}
x_1)\nabla^2v_{k}^\alpha +
\frac{a_{k+1}}{\lambda_{k+1}^2}\bar\Gamma(\lambda_{k+1}x_1)\nabla^2a_{k+1}
+ \frac{1}{\lambda_{k+1}^2}\tbar\Gamma(\lambda_{k+1}x_1)(\nabla a_{k+1})^{\otimes 2}\Big) 
\\ & = - \mathcal{F}_{k+1}
+ \Big(a_{k+1}^2 + \frac{a_{k+1}}{\lambda_{k+1}} \Gamma(\lambda_{k+1}
x_1)\partial_{22} v_{k}^\alpha\\ & \qquad\qquad\qquad -
\frac{a_{k+1}}{\lambda_{k+1}^2}\bar\Gamma(\lambda_{k+1}x_1)\partial_{22}a_{k+1}
- \frac{1}{\lambda_{k+1}^2}\tbar\Gamma(\lambda_{k+1}x_1)(\partial_2a_{k+1})^{2}
\Big) e_2{\otimes} e_2.
\end{split}
\end{equation}

\smallskip

{\bf 3. (Adding the second corrugation)} We define the final 
fields $ v_{k+1} \in\mathcal{C}^\infty(\bar\omega +\bar B_{\delta}(0),\R^3)$,
$w_{k+1}\in\mathcal{C}^\infty(\bar\omega +\bar B_{\delta}(0), \R^2)$ by
using Lemma \ref{lem_step2} (with $i=2, j=\beta$) and the perturbation 
amplitude dictated by (\ref{jamea3}), namely:
\begin{equation}\label{vw_fin}
\begin{split}
& v_{k+1} = \tilde v + \frac{b_{k+1}}{\mu_{k+1}} \Gamma(\mu_{k+1} x_2)e_\beta, 
\\ & w_{k+1} = \tilde w - \frac{b_{k+1}}{\mu_{k+1}} \Gamma(\mu_{k+1} x_2)\nabla v^\beta_k 
+ \frac{b_{k+1}}{\mu_{k+1}}\bar\Gamma(\mu_{k+1}x_2)\nabla b_{k+1}+
\frac{b_{k+1}^2}{\mu_{k+1}}\tbar\Gamma(\mu_{k+1} x_2)e_2,
\\ & \mbox{where }\; 
b_{k+1}^2 = a_{k+1}^2 + \frac{a_{k+1}}{\lambda_{k+1}} \Gamma(\lambda_{k+1}
x_1)\partial_{22} v_{k}^\alpha 
\\ & \qquad\qquad\qquad \;\;
- \frac{a_{k+1}}{\lambda_{k+1}^2}\bar\Gamma(\lambda_{k+1}x_1)\partial_{22}a_{k+1}
- \frac{1}{\lambda_{k+1}^2}\tbar\Gamma(\lambda_{k+1}x_1)(\partial_2a_{k+1})^{2}.
\end{split}
\end{equation}
Firstly, we argue that $b_{k+1}\in\mathcal{C}^\infty(\bar\omega + \bar B_{\delta}(0))$ is well
defined and it satisfies:
\begin{align*}
&  \frac{\tilde C\tilde C_k}{4}\mu_k^\gamma\leq b_{k+1}^2\leq
{2\tilde C \tilde C_k}\mu_k^\gamma \quad\mbox{ and so }\quad
\|b_{k+1}\|\leq C\tilde C_k^{1/2}\mu_k^{\gamma/2}, 
\tag*{(\theequation)$_1$}\refstepcounter{equation} \label{b1}\\
&  \|\nabla^{(m)} b_{k+1}\|_0\leq C\tilde
C_k^{1/2}\mu_k^{\gamma/2}\frac{\lambda_{k+1}^m}{\lambda_{k+1}/\mu_k} 
\quad\mbox{ for all } \; m=1\ldots M+3, \vspace{3mm} \tag*{(\theequation)$_2$}\label{b2} 
\end{align*}
where $C$ in \ref{b1} depends on $\omega,\gamma,N$, and in \ref{b2} on
$\omega,\gamma, N, M$. Indeed by \ref{Ebound13}, \ref{Ebound23}, (\ref{ass_HQ2}),
and (\ref{ass_de2}) we get, with $C$ depending on $\omega,\gamma, N$:
$$\|b_{k+1}^2-a_{k+1}^2\|_0\leq C\big(\tilde
C_k^{1/2}A_k^{1/2}\mu_k^{\gamma/2}\frac{\mu_{k-1}}{\lambda_{k+1}} +
\tilde C_k\mu_k^\gamma \frac{\mu_{k}}{\lambda_{k+1}}\big) 
\leq C \frac{\tilde C_k\mu_k^\gamma}{\lambda_{k+1}/\mu_k} \leq
\frac{\tilde C \tilde C_k}{4}\mu_k^\gamma$$
where the last estimate follows by taking $\lambda_{k+1}/\mu_k$
sufficiently large in function of $\omega,\gamma, N$, which in turn
follows from the third assumption in (\ref{ass_de2}) if $\sigma_0$ has
been assigned large enough. Combined with \ref{Ebound13}, the above
yields \ref{b1}. Towards \ref{b2}, we similarly get, for all
$m=1\ldots M+3$ and $C$ depending on $\omega,\gamma, N, M$:
\begin{equation*}
\begin{split}
& \|\nabla^{(m)}(b_{k+1}^2-a_{k+1}^2)\|_0 \leq C\Big( \sum_{p+q+z=m}
\lambda_{k+1}^{p-1}\|\nabla^{(q)}a_{k+1}\|_0\|\nabla^{(z+2)}v_k^\alpha\|_0
\\ & \qquad\qquad +
\sum_{p+q+z=m}\lambda_{k+1}^{p-2}\big(\|\nabla^{(q)}a_{k+1}\|_0\|\nabla^{(z+2)}a_{k+1}\|_0 
+ \|\nabla^{(q+1)}a_{k+1}\|_0\|\nabla^{(z+1)}a_{k+1}\|_0\Big)
\\ & \leq C \lambda_{k+1}^m\Big( \mu_k^{\gamma/2} \frac{\tilde
C_k^{1/2}A_k^{1/2}}{\lambda_{k+1}/\mu_{k-1}} +
\tilde C_k\frac{\mu_k^\gamma }{\lambda_{k+1}/\mu_{k}}\Big) 
\leq C \tilde C_k\mu_k^{\gamma}\frac{\lambda_{k+1}^m}{\lambda_{k+1}/\mu_k}.
\end{split}
\end{equation*}
Since the same bound is enjoyed by $\nabla^{(m)}a_{k+1}^2$, it follows that:
$$ \|\nabla^{(m)} b_{k+1}^2 \|_0 \leq C \tilde
C_k\mu_k^{\gamma}\frac{\lambda_{k+1}^m}{\lambda_{k+1}/\mu_k}.$$ 
In view of \ref{b1} and via the application of Faa di Bruno's
formula we conclude \ref{b2}:
\begin{equation*}
\begin{split}
\|\nabla^{(m)}b_{k+1}\|_0 & \leq C\Big\|\sum_{p_1+2p_2+\ldots
  mp_m=m} b_{k+1}^{2(1/2-p_1-\ldots -p_m)}\prod_{z=1}^m\big|\nabla^{(z)}b_{k+1}^2\big|^{p_z}\Big\|_0
\\ & \leq C\|b_{k+1}\|_0 \sum_{p_1+2p_2+\ldots
  mp_s=m}\prod_{z=1}^m\Big(\frac{\|\nabla^{(z)}b_{k+1}^2\|_0}{
  \tilde C \tilde C_k \mu^\gamma}\Big)^{p_z}
\leq C \tilde C_k^{1/2}\mu_k^{\gamma/2}\frac{\lambda_{k+1}^m}{\lambda_{k+1}/\mu_k}.
\end{split}
\end{equation*}

\smallskip

\noindent We are now ready to carry out the bounds similar to those done with
the first corrugation in step 2. By \ref{b1}, \ref{b2} we have for all $m=0\ldots M$:
\begin{equation}\label{jamea00}
\begin{split}
& \|(\tilde v - v_{k+1})^\beta\|_1 \leq C \big(\|b_{k+1}\|_0 +
\frac{\|\nabla b_{k+1}\|_0}{\mu_{k+1}}\big)  \leq C\tilde C_k^{1/2}\mu_k^{\gamma/2}\\ 
& \|\nabla^{(m+2)}(\tilde v - v_{k+1})^\beta\|_0 \leq C \sum_{p+q=m+2}\mu_{k+1}^{p-1}
\|\nabla^{(q)}b_{k+1}\|_0 \leq C\tilde C_k^{1/2}\mu_k^{\gamma/2}\mu_{k+1}^{m+1},
\end{split}
\end{equation}
where $C$ in the first bound depends on $\omega,\gamma, N$, while in
the second bound on $\omega,\gamma,N, M$. At this point, recalling
(\ref{jamea0}) we conclude \ref{P2bound1}, because
$v_{k+1}^\alpha=\tilde v^\alpha$ and $v_k^\beta = \tilde v^\beta$.
We likewise get \ref{P2bound2}, recalling (\ref{jamea1}) and deducing
from (\ref{jamea00}) and (\ref{ass_HQ2}) that:
$$\|\nabla^{(m+2)}v_{k+1}^\beta\|_0 \leq C \tilde C_k^{1/2}\mu_k^{\gamma/2}\mu_{k+1}^{m+1} +
B_k^{1/2}\lambda_{k}^{m+1} \leq C\tilde C_k^{1/2}\mu_k^{\gamma/2}\mu_{k+1}^{m+1},$$
where in the last bound we used (\ref{ass_de2}) to observe that:
$$ B_{k}^{1/2} \lambda_k \leq \tilde C_k^{1/2}\mu_k
\frac{\mu_{k+1}}{\lambda_{k+1}} \leq \tilde C_k^{1/2} \mu_{k+1}.$$
Finally, we observe that the formula for $w_{k+1}-\tilde w$ is exactly
the same as that for $w_k-\tilde w$, where $a_{k+1}$ is exchanged for
$b_{k+1}$, $\lambda_{k+1}$ for $\mu_{k+1}$ and $\nabla v_k^\alpha$ to
$\nabla v_k^\beta$, and minus the term $\Psi_{k+1}$. Hence, the same
estimate as in (\ref{jamea2}) now yields, in view of \ref{b1},
\ref{b2}, (\ref{ass_HQ2}):
\begin{equation*}
\begin{split}
& \|\tilde w - w_{k+1}\|_1\leq C \tilde C_k^{1/2}\mu_k^{\gamma/2} 
\big(\|\nabla v_k^\beta\|_0 + \tilde C_k^{1/2}\mu_k^{\gamma/2}\big),\\
& \|\nabla^2 (\tilde w - \bar  w_{k+1})\|_0\leq 
C \tilde C_k^{1/2}\mu_k^{\gamma/2} \big(\|\nabla v_k^\beta\|_0 + \tilde
C_k^{1/2}\mu_k^{\gamma/2}\big)\mu_{k+1,}
\end{split}
\end{equation*}
with constants $C$ depending on $\omega,\gamma, N$. Recalling
(\ref{jamea2}), the above yields \ref{P2bound3}, where we note that we
necessitated bounds on $\|v_k^\beta\|_3$ and $\|b_{k+1}\|_3$.

\smallskip

{\bf 4. (The bound on the derived deficit)} It remains to show \ref{P2bound4}.
To this end, we use Lemma \ref{lem_step2} and recall that
$\tilde{\mathcal{D}} = -\mathcal{F}_{k+1} + b_{k+1}^2 
e_2\otimes e_2$ from (\ref{jamea3}), to write:
\begin{equation*}
\begin{split}
&\mathcal{D}_{k+1}= \tilde{\mathcal{D}} - b_{k+1}^2e_2\otimes e_2 -
\mathcal{S}_{k+1} = -\mathcal{F}_{k+1} - \mathcal{S}_{k+1},\\
& \mbox{where }\; 
\mathcal{S}_{k+1} = -\frac{b_{k+1}}{\mu_{k+1}} \Gamma(\mu_{k+1} x_2)\nabla^2 v_{k}^\beta
+ \frac{b_{k+1}}{\mu_{k+1}^2}\bar\Gamma(\mu_{k+1} x_2)\nabla^2 b_{k+1}
+ \frac{1}{\mu_{k+1}^2}\dbar\Gamma(\mu_{k+1}x_2)(\nabla b_{k+1})^{\otimes 2}.
\end{split}
\end{equation*}
We now use (\ref{ass_HQ2}) and \ref{b1}, \ref{b2} to get, for all $m=0\ldots M$:
\begin{equation}\label{kalb}
\begin{split}
& \|\nabla^{(m)}\mathcal{S}_{k+1}\|_0\leq C\Big( \sum_{p+q+z=m}\hspace{-3mm}
\mu_{k+1}^{p-1}\|\nabla^{(q)}b_{k+1}\|_0\|\nabla^{(z+2)}v_k^\beta\|_0
\\ & \qquad\qquad +
\sum_{p+q+z=m}\hspace{-3mm}\mu_{k+1}^{p-2}\big(\|\nabla^{(q)}b_{k+1}\|_0\|\nabla^{(z+2)}b_{k+1}\|_0 
+ \|\nabla^{(q+1)}b_{k+1}\|_0\|\nabla^{(z+1)}b_{k+1}\|_0\big)\Big)
\\ & \leq C \sum_{p+q+z=m}\hspace{-2mm}\Big( \mu_{k+1}^{p-1} \tilde C_k^{1/2}\mu_k^{\gamma/2}
\lambda_{k+1}^qB_k^{1/2}\lambda_k^{z+1} 
+ \mu_{k+1}^{p-2}\tilde C_k\mu_k^\gamma\frac{\lambda_{k+1}^{q+z+2}}{\lambda_{k+1}/\mu_k}\Big)
\\ & \leq C \mu_{k+1}^m\Big( \mu_k^{\gamma/2} \frac{\tilde
C_k^{1/2}B_k^{1/2}}{\mu_{k+1}/\lambda_{k}} +
\tilde C_k\mu_k^\gamma\frac{1}{(\mu_{k+1}/\lambda_{k+1})^2(\lambda_{k+1}/\mu_{k})}\Big) 
\leq C \tilde C_k\mu_k^{\gamma}\frac{\mu_{k+1}^m}{(\lambda_{k+1}/\mu_k)^N},
\end{split}
\end{equation}
where the last bound follows from the last assumption in
(\ref{ass_de2}) because:
\begin{equation*}
\begin{split}
& \frac{B_k^{1/2}}{\mu_{k+1}/\lambda_k} \leq\tilde C_k^{1/2}
\frac{\mu_k}{\lambda_k} \frac{\mu_{k+1}}{\lambda_{k+1}}
\frac{1}{(\lambda_{k+1}/\mu_k)^{N-1}} \leq  \frac{1}{(\lambda_{k+1}/\mu_k)^N} 
\\ & \mbox{and } \; \big(\frac{\mu_{k+1}}{\lambda_{k+1}}\big)^2 \frac{\lambda_{k+1}}{\mu_k}
\geq \big(\frac{\lambda_{k+1}}{\mu_k}\big)^{N-1} \frac{\lambda_{k+1}}{\mu_k}
= \big(\frac{\lambda_{k+1}}{\mu_k}\big)^N.
\end{split}
\end{equation*}
Recalling \ref{Ebound43}, the bound in (\ref{kalb}) yields:
\begin{equation*}
\|\nabla^{(m)}\mathcal{D}_{k+1} \|_0 
\leq C \tilde C_k\mu_k^{\gamma}\lambda_{k+1}^{\gamma N}
\frac{\mu_{k+1}^m}{(\lambda_{k+1}/\mu_k)^N}.
\end{equation*}
for all $m=0\ldots M$ and with $C$ depending on $\omega, \gamma, N,
M$. We note that for carrying out the whole argument, we necessitated
bounds on $a_{k+1}$ and $v_k^\alpha$ up to $\max\{M+4,5\}$
derivatives, on $v_k^\beta$ up to $\max\{M+2,3\}$ derivatives, and on
$\mathcal{F}_{k+1}$ up to $M$ derivatives. This justifies the
derivative count in (\ref{ass_HQ2}).
The proof is done.
\end{proof}

\medskip

We are now ready to present the main result of this
section, in which Proposition \ref{prop2} is iterated with the
consecutive choice of $\alpha_k,\beta_k$ in the mod $3$ arithmetic. 
This construction necessitates
specific assumptions on the ratios of the employed
frequencies. The viability of these assumptions and existence of
frequencies satisfying them will be shown in the next section. 

\begin{proposition}\label{prop3}
Let $\omega\subset\R^2$ be open, bounded, smooth and let $N,K\geq 1$
and $\gamma\in (0,1)$. Then, there exists $l_0\in (0,1)$
depending only on $\omega$, and $\sigma_0\geq 1$ depending on $\omega,
\gamma, N, K$ such that the following holds. Given the positive constants
$l,\eta, \mu_0$ such that:
\begin{equation}\label{ass_de3}
l+K\eta\leq 2l_0,\qquad \mu_0\geq \frac{1}{\eta},
\end{equation}
and given  $v_{0}\in\mathcal{C}^\infty(\bar\omega+\bar
B_{l+K\eta}(0),\R^{3})$, $w_{0}\in\mathcal{C}^\infty(\bar\omega+\bar
B_{l+K\eta}(0),\R^{2})$, $A_0\in\mathcal{C}^\infty(\bar\omega+\bar
B_{l+K\eta}(0),\R^{2\times 2}_\sym)$, such that together with the derived field $\mathcal{D}_{0}
= A_0 - \big(\frac{1}{2} (\nabla v_0)^T\nabla v_0 +\sym\nabla w_0\big)$, the
following bounds hold with some auxiliary constant $\tilde C_0>0$:
\begin{equation}\label{ass_HQ3}
\begin{array}{l}
\displaystyle{\|\nabla^{(m)}\mathcal{D}_0\|_0\leq \tilde C_0\mu_0^m} \vspace{2mm}\\
\displaystyle{\|\nabla^{(m+2)}v_0\|_0\leq \tilde C_0^{1/2}\mu_{0}^{m+1} }
\end{array}
\qquad \mbox{for all } \; m=0\ldots 3K(N+1),
\end{equation}
one can apply Proposition \ref{prop2} consecutively for $k=0\ldots K-1$ with:
\begin{equation}\label{albe} 
\alpha_k= (2k)\; \mathrm{mod}\;  3+1, \qquad \beta_k= (2k+1) \;\mathrm{mod}\; 3 +1
\end{equation}
and with the frequencies $\{\lambda_k, \mu_k\}_{k=1}^K$ as long as the
following conditions are satisfied:
\begin{equation}\label{ass_ml}
\begin{split}
& \lambda_{k+1}^{1-\gamma}\geq \mu_k\sigma_0 \quad \mbox{ for all }\; k =0\ldots K-1, \qquad
\frac{\mu_1}{\lambda_1} \geq \big(\frac{\lambda_1}{\mu_0}\big)^N,
\\ & \frac{\mu_k}{\lambda_{k}}\geq 
\max\Big\{\big(\frac{\lambda_{k}}{\mu_{k-1}}\frac{\lambda_{k-1}}{\mu_{k-2}}\big)^{N/2},
\frac{\big(\lambda_{k}/\mu_{k-1}\big)^{N}
  (\lambda_{k-1}/\mu_{k-2})^{N/2}}{\mu_{k-1}/\lambda_{k-1}}\Big\} 
\quad \mbox{ for all } \; k =2\ldots K-1,\\
& \frac{\mu_K}{\lambda_K}\geq
\max\Big\{\big(\frac{\lambda_{K}}{\mu_{K-1}}\big)^{N/2}, 
\frac{\big(\lambda_{K}/\mu_{K-1}\big)^{N} (\lambda_{K-1}/\mu_{K-2})^{N/2}}{\mu_{K-1}/\lambda_{K-1}}\Big\},
\end{split}
\end{equation}
to obtain $\tilde v \in\mathcal{C}^\infty(\bar\omega+\bar
B_{l}(0),\R^{3})$ and $\tilde w\in\mathcal{C}^\infty(\bar\omega+\bar
B_{l}(0),\R^{2})$, such that denoting the new derived field $\tilde{\mathcal{D}}
= A_0 - \big(\frac{1}{2} (\nabla \tilde v)^T\nabla \tilde v
+\sym\nabla \tilde w\big)$, the following bounds hold:
\begin{align*}
& \begin{array}{l}
\displaystyle{\|\tilde v - v_0\|_1\leq C \tilde C_0^{1/2} \Lambda^{\gamma/2}},
\end{array}\vspace{4mm} 
\tag*{(\theequation)$_1$}\refstepcounter{equation} \label{P3bound1}\\
& \begin{array}{l} \displaystyle{\|\nabla^{2}\tilde v\|_0\leq C\tilde
  C_0^{1/2}\Lambda^{\gamma/2}\mu_K\Big(\prod_{k=0}^{K-2}\frac{1}{(\lambda_{k+1}/\mu_k)^{N/2}}\Big)
\Big(1+ \frac{(\lambda_{K-1}/\mu_{K-2})^{N/2} } {\mu_{K}/\mu_{K-1}}\Big)},
\end{array}
\vspace{3mm} \tag*{(\theequation)$_2$} \label{P3bound2}\\
& \begin{array}{l} 
\displaystyle{\|\tilde w - w_0\|_1\leq C\tilde
  C_0^{1/2}\Lambda^\gamma\big(\|\nabla v_0\|_0 + \tilde
  C_0^{1/2}\big)} \vspace{2mm} \\
\displaystyle{\|\nabla^2(\tilde w - w_0)\|_0\leq C\tilde
  C_0^{1/2}\Lambda^{\gamma}\mu_K\Big(\prod_{k=0}^{K-2}\frac{1}{(\lambda_{k+1}/\mu_k)^{N/2}}
\Big)}\times\vspace{2mm}\\
\qquad\qquad \qquad\qquad \qquad \qquad\quad 
\times \displaystyle{\Big(1+ \frac{(\lambda_{K-1}/\mu_{K-2})^{N/2} }
  {\mu_{K}/\mu_{K-1}}\Big) \big(\|\nabla 
v_0\|_0+\tilde C_0^{1/2}\big)},
\end{array}
\vspace{3mm} \tag*{(\theequation)$_3$} \label{P3bound3}\\
& \begin{array}{l} \displaystyle{\|\tilde{\mathcal{D}}\|_0\leq C\tilde
C_0\Lambda^{\gamma}\prod_{k=0}^{K-1} \frac{1}{(\lambda_{k+1}/\mu_k)^N}},
\end{array}
\tag*{(\theequation)$_4$} \label{P3bound4}
\end{align*}
Above, constants $C$ depend on: $\omega, \gamma, N, K$, and we denoted:
$\Lambda = \prod_{k=0}^K(\mu_k\lambda_k^N)$.
\end{proposition}
\begin{proof}
{\bf 1. (Setting the inductive quantities)}
For $k=0\ldots K-1$ we will define the fields:
\begin{equation*}
\begin{split}
& v_{k+1}\in\mathcal{C}^\infty(\bar\omega + \bar B_{l+(K-(k+1))\eta}, \R^3),
\quad w_{k+1}\in\mathcal{C}^\infty(\bar\omega + \bar B_{l+(K-(k+1))\eta}, \R^2),
\\ & \mathcal{D}_{k+1}
= A_0 - \big(\frac{1}{2} (\nabla v_{k+1})^T\nabla v_{k+1} +\sym\nabla w_{k+1}\big),
\end{split}
\end{equation*}
by applying Proposition \ref{prop2} to the previous fields
$v_{k}\in\mathcal{C}^\infty(\bar\omega + \bar B_{l+(K-k)\eta},
\R^3)$, $w_{k}\in\mathcal{C}^\infty(\bar\omega + \bar B_{l+(K-k)\eta}, \R^2)$
with the fixed parameters $N$, $\gamma$, the
given frequencies $\mu_{k-1}\leq\lambda_k \leq \mu_k\leq \lambda_{k+1}\leq\mu_{k+1}$
the inductive parameters $\delta_k,\eta_k$ satisfying
$\delta_k+\eta_k\leq 2l_0$ and $M_k\geq 0$ where:
$$\mu_{-1}=\lambda_0=\mu_0, \qquad \delta_k = l+ (K-(k+1))\eta,\quad
\eta_k = \eta,\qquad M_k = 3(K-(k+1))(N+1),$$
and the auxiliary constants $\tilde C_k, A_k, B_k$ defined as follows:
\begin{equation}\label{CAB}
\begin{split}
& A_1=A_0=B_0=\tilde C_0,\\
& \tilde C_{k+1} = C \tilde C_k
\frac{\mu_k^\gamma\lambda_{k+1}^{\gamma N}}{(\lambda_{k+1}/\mu_k)^N}
\qquad \mbox{ for all }\; k=0\ldots K-1,\\
& A_k=C \tilde C_{k-2}\mu_{k-2}^\gamma 
\qquad \qquad  \quad \, \mbox{ for all }\; k=2\ldots K+1, \\
& B_k=C \tilde C_{k-1}\mu_{k-1}^\gamma \qquad \qquad \quad \, \mbox{ for all }\; k=1\ldots K,
\end{split}
\end{equation}
where constants $C$ depend on $\omega, \gamma, N, K, k$. The
codimension components $\alpha_k, \beta_k$ are as in (\ref{albe}) and
we call $\gamma_k \in \{1,2,3\}\setminus \{\alpha_k, \beta_k\}$ the
remaining component:
$$ \gamma_k= \alpha_{k+1} = (2k+2)\; \mathrm{mod}\;  3+1.$$
The applicability of Proposition \ref{prop2} relies on the following
bounds, that will be validated throughout the proof for all $k=0\ldots K$:
\begin{equation}\label{ABC}
\begin{split}
& \|\nabla^{(m)}\mathcal{D}_k\|_0 \leq \tilde C_k\mu_k^m,
\qquad\qquad\qquad \,\|\nabla^{(m+2)}v_k^{\alpha_k}\|_0\leq A_k^{1/2}\mu_{k-1}^{m+1}, \\
& \|\nabla^{(m+2)}v_k^{\beta_k}\|_0\leq
B_k^{1/2}\lambda_{k}^{m+1},\qquad \quad \;
\|\nabla^{(m+2)}v_k^{\gamma_k}\|_0\leq A_{k+1}^{1/2}\mu_{k}^{m+1}. 
\end{split}
\end{equation}
Eventually, we will set $\tilde v=v_K$ and $\tilde w = w_K$ and deduce
the bounds \ref{P3bound1} -- \ref{P3bound4}.

\smallskip

{\bf 2. (Induction base $k=0$ and $k=1$)} 
We have $\alpha_0=1$, $\beta_0=2$, $\gamma_0=3$ and so (\ref{ABC}) holds with $k=0$.
We now check conditions in (\ref{ass_de2}). The first five conditions
are valid by assumption, while the last one becomes:
$\mu_1/\lambda_1\geq (\lambda_1/\mu_0)^{N-1}$ and it is implied by the
third assumption in (\ref{ass_ml}). Consequently, Proposition \ref{prop2} yields:
\begin{equation*}
\begin{split}
& \|\nabla^{(m)}\mathcal{D}_1\|_0 \leq C \tilde C_0\mu_0^{\gamma}
\lambda_1^{\gamma N} \frac{\mu_1^m}{(\lambda_1/\mu_0)^N} \doteq \tilde
C_1\mu_1^m,\quad 
\|\nabla^{(m+2)}v_1^{1}\|_0\leq C\tilde C_0^{1/2}
\mu_0^{\gamma/2}\lambda_1^{m+1} \doteq B_1^{1/2}\lambda_{1}^{m+1}, \\
& \|\nabla^{(m+2)}v_1^{2}\|_0\leq C\tilde C_0^{1/2} \mu_0^{\gamma/2}\mu_1^{m+1}
\doteq A_2^{1/2}\mu_{1}^{m+1},\quad \,
\|\nabla^{(m+2)}v_1^{3}\|_0\leq \tilde C_0^{1/2}\mu_0^{m+1}\leq A_{1}^{1/2}\mu_{0}^{m+1}. 
\end{split}
\end{equation*}
where constants $C$ depend on $\omega,\gamma, N, K$ and where we
defined the new quantities $\tilde C_1, B_1, A_2$ according to
(\ref{CAB}). The above bounds are exactly (\ref{ABC}) at $k=1$.

\smallskip

\noindent Continuing, for $k=1$ we have $\alpha_1=3$, $\beta_1=1$,
$\gamma_1=2$. We need to the check the last two conditions in
(\ref{ass_de2}). This is validated as follows, in virtue of (\ref{ass_ml}):
\begin{equation*}
\begin{split}
& \big(\frac{A_1}{\tilde C_1}\big)^{1/2} =
\frac{(\lambda_1/\mu_0)^{N/2}}{C\mu_0^{\gamma/2} \lambda_1^{\gamma N/2}} \leq
\big(\frac{\lambda_1}{\mu_0}\big)^{N/2}\leq
\big(\frac{\mu_1}{\lambda_1}\big)^{1/2} \leq \frac{\mu_1}{\mu_0},\\
& \big(\frac{\lambda_2}{\mu_1}\big)^{(N-1)/2}\leq \big(\frac{\lambda_2}{\mu_1}\big)^{N/2}
\leq \big(\frac{\lambda_2}{\mu_1}\frac{\lambda_1}{\mu_0}\big)^{N/2} \leq
\frac{\mu_2}{\lambda_2},\\ 
& \big(\frac{\lambda_2}{\mu_1}\big)^{N-1} \frac{(B_1/\tilde C_1)^{1/2}}{\mu_1/\lambda_1}=
\frac{(\lambda_2/\mu_1)^{N-1}}{\mu_1/\lambda_1} \big(\frac{C
  (\lambda_1/\mu_0)^N}{\lambda_1^{\gamma N}}\big)^{1/2}
\leq \frac{(\lambda_2/\mu_1)^{N}(\lambda_1/\mu_0)^{N/2}}{\mu_1/\lambda_1} 
\frac{C}{\lambda_2/\mu_1}\leq \frac{\mu_2}{\lambda_2},
\end{split}
\end{equation*}
provided that $\sigma_0$ has been chosen large enough (in function of
$\omega, \gamma, N, K$) to ensure that the last quotient above is less
than $1$ in virtue of the first assumption in (\ref{ass_ml}) at $k=1$.
Consequently, Proposition \ref{prop2} implies:
\begin{equation*}
\begin{split}
& \|\nabla^{(m)}\mathcal{D}_2\|_0 \leq C \tilde C_1\mu_1^{\gamma}
\lambda_2^{\gamma N} \frac{\mu_2^m}{(\lambda_2/\mu_1)^N} \doteq \tilde
C_2\mu_2^m,\quad 
\|\nabla^{(m+2)}v_2^{3}\|_0\leq C\tilde C_1^{1/2} 
\mu_1^{\gamma/2}\lambda_2^{m+1} \doteq B_2^{1/2}\lambda_{2}^{m+1}, \\
& \|\nabla^{(m+2)}v_2^{1}\|_0\leq C\tilde C_0^{1/2} \mu_1^{\gamma/2}\mu_2^{m+1}
\doteq A_3^{1/2}\mu_{2}^{m+1}, \;\;
\|\nabla^{(m+2)}v_2^{2}\|_0 = \|\nabla^{(m+2)}v_1^{2}\|_0
\leq A_{2}^{1/2}\mu_{1}^{m+1},
\end{split}
\end{equation*}
where constants $C$ depend on $\omega,\gamma, N, K$ and where we
relied on the definition (\ref{CAB}). The above bounds are exactly (\ref{ABC}) at $k=2$.

\smallskip

{\bf 3. (Induction step)} Let now $k=2\ldots K-1$ and assume
(\ref{ABC}). To verify the fifth assumption in (\ref{ass_de2}), we
use (\ref{CAB}) and (\ref{ass_ml}) in:
\begin{equation*}
\begin{split}
\big(\frac{A_k}{\tilde C_k}\big)^{1/2} & =
\Big(\frac{C \tilde C_{k-2}\mu_{k-2}^\gamma}{\tilde C_k}\Big)^{1/2}=
\Big(\frac{C \mu_{k-2}^\gamma}{(\tilde C_k/\tilde C_{k-1}) 
(\tilde C_{k-1}/\tilde C_{k-2}) }\Big)^{1/2} \\ & = 
\Big(\frac{C (\lambda_k/\mu_{k-1})^N(\lambda_{k-1}/\mu_{k-2})^N }
{\mu_{k-1}^\gamma\lambda_k^{\gamma N}\lambda_{k-1}^{\gamma N}}\Big)^{1/2}
\leq C \big(\frac{\lambda_k}{\mu_{k-1}}
\frac{\lambda_{k-1}}{\mu_{k-2}}\big)^{N/2} \leq C\frac{\mu_k}{\lambda_k}
\leq \frac{\mu_k}{\mu_{k-1}},
\end{split}
\end{equation*}
because $\lambda_k/\mu_{k-1}$ is larger than the constant $C$
depending on $\omega,\gamma,N,K$ provided that $\sigma_0$ in the first
assumption in (\ref{ass_ml}) is sufficiently large. For the last
assumption in (\ref{ass_HQ2}), we first observe:
\begin{equation*}
\begin{split}
\big(\frac{\lambda_{k+1}}{\mu_k}\big)^{N-1} 
\frac{(B_k/\tilde C_{k})^{1/2}}{\mu_k/\lambda_k} & =
\frac{(\lambda_{k+1}/\mu_k)^{N-1}}{\mu_k/\lambda_k}
\Big(\frac{C\mu_{k-1}^\gamma}{\tilde C_k/\tilde C_{k-1}}\Big)^{1/2}
= \frac{(\lambda_{k+1}/\mu_k)^{N-1}}{\mu_k/\lambda_k}
\Big(\frac{C (\lambda_k/\mu_{k-1})^N}{\lambda_k^{\gamma N}}\Big)^{1/2}
\\ & \leq \frac{(\lambda_{k+1}/\mu_k)^{N}(\lambda_k/\mu_{k-1})^{N/2}}{\mu_k/\lambda_k} 
\frac{C}{\lambda_{k+1}/\mu_k}\leq \frac{\mu_{k+1}}{\lambda_{k+1}},
\end{split}
\end{equation*}
where in the last bound above we used (\ref{ass_ml}) and, again the
fact that $\lambda_{k+1}/\mu_{k}\geq C$ is $\sigma_0$ if sufficiently
large. It remains to estimate $(\lambda_{k+1}/\mu_k)^{(N-1)/2}$. In
case of $k\leq K-2$, we get:
$$\big(\frac{\lambda_{k+1}}{\mu_k}\big)^{(N-1)/2}\leq 
\big(\frac{\lambda_{k+1}}{\mu_k}\big)^{N/2} 
\leq \big(\frac{\lambda_{k+1}}{\mu_k}\frac{\lambda_k}{\mu_{k-1}}\big)^{N/2} 
\leq \frac{\mu_{k+1}}{\lambda_{k+1}}$$
by (\ref{ass_ml}). Likewise, for $k=K-1$:
$$\big(\frac{\lambda_{k+1}}{\mu_k}\big)^{(N-1)/2}\leq 
\big(\frac{\lambda_{K}}{\mu_{K-1}}\big)^{N/2} 
\leq \frac{\mu_{K}}{\lambda_{K}} =
\frac{\mu_{k+1}}{\lambda_{k+1}}.$$
This ends the verification of (\ref{ass_de2}). 
We may now apply Proposition \ref{prop2} to get:
\begin{equation*}
\begin{split}
& \|\nabla^{(m)}\mathcal{D}_{k+1}\|_0 \leq C \tilde C_k\mu_k^{\gamma}
\lambda_{k+1}^{\gamma N} \frac{\mu_{k+1}^m}{(\lambda_{k+1}/\mu_k)^N} \doteq \tilde
C_{k+1}\mu_{k+1}^m, \\ 
& \|\nabla^{(m+2)}v_{k+1}^{\beta_{k+1}}\|_0=
\|\nabla^{(m+2)}v_{k+1}^{\alpha_k}\|_0\leq C\tilde C_k^{1/2} 
\mu_k^{\gamma/2}\lambda_{k+1}^{m+1} \doteq B_{k+1}^{1/2}\lambda_{k+1}^{m+1}, \\
& \|\nabla^{(m+2)}v_{k+1}^{\gamma_{k+1}}\|_0=
\|\nabla^{(m+2)}v_{k+1}^{\beta_k}\|_0\leq C\tilde C_k^{1/2} \mu_k^{\gamma/2}\mu_{k+1}^{m+1}
\doteq A_{k+2}^{1/2}\mu_{k+1}^{m+1},\\ 
& \|\nabla^{(m+2)}v_{k+1}^{\alpha_{k+1}}\|_0=
\|\nabla^{(m+2)}v_{k+1}^{\gamma_k}\|_0 = \|\nabla^{(m+2)}v_{k}^{\gamma_k}\|_0
\leq A_{k+1}^{1/2}\mu_{k}^{m+1}. 
\end{split}
\end{equation*}
where we used (\ref{CAB}), to obtain (\ref{ABC}) at $k+1$.

\smallskip

{\bf 4. (Gathering bounds on $v_K$, $w_K$)} Recall that we have set
$\tilde v=v_K$, $\tilde w= w_K$ and hence $\tilde{\mathcal{D}} =
\mathcal{D}_K$. Applying (\ref{ABC}) at
the final counter value $K$ and recalling the progression in
(\ref{CAB}), we obtain the following bounds, in which $C$ depend on
$\omega,\gamma, N, K$:
\begin{equation*}
\|\mathcal{D}_K\|_0\leq \tilde C_K \leq C\tilde C_0\prod_{k=0}^{K-1}
\frac{\mu_k^\gamma\lambda_{k+1}^{\gamma N}}{(\lambda_{k+1}/\mu_k)^N}
\leq C\tilde C_0\Lambda^\gamma \prod_{k=0}^{K-1}
\frac{1}{(\lambda_{k+1}/\mu_k)^N},
\end{equation*}
which is exactly \ref{P3bound1}. Further, by (\ref{CAB}) we get:
\begin{equation}\label{shaq}
\begin{split}
\|\nabla^2v_K\|_0& \leq A_K^{1/2}\mu_{K-1}+B_K^{1/2}\lambda_K +
A_{K+1}^{1/2}\mu_K \leq C\big(\tilde C_{K-2}^{1/2}\mu_{K-2}^{\gamma/2}\mu_{K-1}
+ \tilde C_{K-1}^{1/2}\mu_{K-1}^{\gamma/2}\mu_K\big)
\\ & \leq C\tilde C_0^{1/2}\Big(\Big(\prod_{k=0}^{K-3}
\frac{\mu_k^{\gamma/2}\lambda_{k+1}^{\gamma N/2}}{(\lambda_{k+1}/\mu_k)^{N/2}}\Big)
\mu_{K-2}^{\gamma/2}\mu_{K-1} +
\Big(\prod_{k=0}^{K-2}
\frac{\mu_k^{\gamma/2}\lambda_{k+1}^{\gamma N/2}}{(\lambda_{k+1}/\mu_k)^{N/2}}\Big)
\mu_{K-1}^{\gamma/2}\mu_{K} \Big) 
\\ & \leq C\tilde C_0^{1/2} \Lambda^{\gamma/2}\Big(\prod_{k=0}^{K-2}
\frac{1}{(\lambda_{k+1}/\mu_k)^{N/2}}\Big)\Big(
1+\frac{(\lambda_{K-1}/\mu_{K-2})^{N/2}}{\mu_K/\mu_{K-1}}\Big)\mu_{K},
\end{split}
\end{equation}
which is \ref{P3bound2}. To obtain \ref{P3bound1}, we sum the
estimates \ref{P2bound1} in:
\begin{equation*}
\|v_K-v_0\|_1 \leq \sum_{k=0}^{K-1}\|v_{k+1}-v_k\|_1\leq C \sum_{k=0}^{K-1}
\tilde C_k^{1/2}\mu_k^{\gamma/2} \leq C\tilde C_0^{1/2}\prod_{k=0}^{K-1}\mu_k^{\gamma/2}
\leq C\tilde C_0^{1/2}\Lambda^{\gamma/2},
\end{equation*}
because for all $k=1\ldots K-1$ there holds: $\tilde C_{k+1} \leq
C \tilde C_k\mu_k^\gamma /\sigma_0^N$ by (\ref{ass_ml}). Similarly,
there follows the first bound in \ref{P3bound3}, in view of
\ref{P2bound3} and the previous estimate:
\begin{equation*}
\begin{split}
\|w_K-w_0\|_1 & \leq \sum_{k=0}^{K-1}\|w_{k+1}-w_k\|_1\leq C \sum_{k=0}^{K-1}
\tilde C_k^{1/2}\mu_k^{\gamma/2} \big(\|\nabla v_k\|_0+\tilde C_k^{1/2}
\mu_k^{\gamma/2}\big)
\\ & \leq C\tilde C_0^{1/2}\Big(\prod_{k=0}^{K-1}\mu_k^{\gamma/2}\Big)
\Big(\|\nabla v_0\|_0 + \tilde C_0^{1/2} \prod_{k=0}^{K-1}\mu_k^{\gamma/2}\Big)
\leq C\tilde C_0^{1/2}\Lambda^{\gamma}\big(\|\nabla v_0\| +\tilde
C_0^{1/2}\big).
\end{split}
\end{equation*}
To complete the claim in \ref{P3bound3}, we use the second bound in
\ref{P2bound3} to get:
\begin{equation}\label{jadid}
\begin{split}
\|\nabla^2(w_K-w_0)\|_0  & \leq
\sum_{k=0}^{K-1}\|\nabla^2(w_{k+1}-w_k)\|_0 \\ & \leq C \sum_{k=0}^{K-1}
\tilde C_k^{1/2}\mu_k^{\gamma/2} \mu_{k+1}\big(\|\nabla v_k\|_0+\tilde C_k^{1/2}
\mu_k^{\gamma/2}\big)
\\ & \leq C \sum_{k=0}^{K-1}\tilde C_k^{1/2}\mu_k^{\gamma/2} \mu_{k+1}
\Big(\|\nabla v_0\|_0 + \tilde C_0^{1/2} \prod_{k=0}^{K-1}\mu_k^{\gamma/2}\Big)
\\ & \leq C \big(\|\nabla v_0\|_0 +\tilde C_0^{1/2}\big) \Big(\prod_{k=0}^{K-1}\mu_k^{\gamma/2}\Big)
\sum_{k=0}^{K-1} \tilde C_k^{1/2}\mu_k^{\gamma/2} \mu_{k+1}.
\end{split}
\end{equation}
Observe now that $\{\tilde
C_k^{1/2}\mu_{k+1}\mu_k^{\gamma/2}\}_{k=0}^{K-2}$ is a nonincreasing
sequence, because from (\ref{CAB}):
$$\frac{\tilde C_{k-1}^{1/2}\mu_k\mu_{k-1}^{\gamma/2}}{\tilde
  C_{k}^{1/2}\mu_{k+1}\mu_{k}^{\gamma/2}} =
C\frac{(\lambda_k/\mu_{k-1})^{N/2}}
{\mu_k^{\gamma/2}\lambda_k^{\gamma N/2}(\mu_{k+1}/\mu_k)}
\leq \frac{C}{\lambda_{k+1}/\mu_k}
\frac{(\lambda_k/\mu_{k-1})^{N/2}}{\mu_{k+1}/\lambda_{k+1}}\leq 1
\quad\mbox{ for all }\; k=1\ldots K-2.$$
Indeed, the first term in the right hand side above is less than $1$ for $\sigma_0$
large enough by the first condition in (\ref{ass_ml}), while the
second term is less than $1$ by the second and third conditions in (\ref{ass_ml}).
Consequently, using the calculation in (\ref{shaq}), estimates in
(\ref{jadid}) yield:
\begin{equation*}
\begin{split}
\|\nabla^2(w_K&-w_0)\|_0  \leq 
C \big(\|\nabla v_0\|_0 +\tilde C_0^{1/2}\big) \Big(\prod_{k=0}^{K-1}\mu_k^{\gamma/2}\Big)
\big(\tilde C_{K-2}^{1/2} \mu_{K-2}^{\gamma/2} \mu_{K-1} + \tilde
C_{K-1}^{1/2} \mu_{K-1}^{\gamma/2} \mu_{K}\big) 
\\ & \leq C\tilde C_0 ^{1/2}\big(\|\nabla v_0\|_0 +\tilde C_0^{1/2}\big) 
\Lambda^{\gamma}\Big(\prod_{k=0}^{K-2} \frac{1}{(\lambda_{k+1}/\mu_k)^{N/2}}\Big)
\Big(1+ \frac{(\lambda_{K-1}/\mu_{K-2})^{N/2}}{\mu_{K}/\mu_{K-1}}\Big) \mu_{K}.
\end{split}
\end{equation*}
This is exactly the second bound in \ref{P3bound3}. The proof is done.
\end{proof}

\section{The Fibonacci frequencies and a proof of Theorem \ref{thm_stage}}
\label{sec_fib}

In this section we complete the inductive procedure put forward in
Proposition \ref{prop3}, by specifying a 	progression of frequencies
which satisfy all the conditions in (\ref{ass_ml}). We will then prove
the single ``stage'' estimates in Theorem \ref{thm_stage}, which will
yield Theorem \ref{th_final}, as described in the next section.
We use the Fibonacci sequence $\{F_k\}_{k=0}^\infty$, where:
$$F_0=F_1=1,\qquad F_{k+2} = F_{k} + F_{k+1} \quad\mbox{ for all } \;
k\geq 0.$$

\begin{proposition}\label{prop4}
Let $N\geq 1$, $K\geq 4$ and let $\gamma\in (0,1)$ satisfy:
\begin{equation}\label{gam_small}
\gamma\leq  \frac{1}{(F_{K+2}-3)(1+N/2)}.
\end{equation}
Then, for every parameters $\mu_0,  \sigma_0,\sigma\geq 1$ such that:
\begin{equation}\label{ass_f}
\frac{\sigma}{(\mu_0\sigma)^\gamma}\geq \sigma_0,
\end{equation}
conditions (\ref{ass_ml}) are satisfied for the sequence
$\{\lambda_k,\mu_k\}_{k=1}^K$ defined in:
\begin{equation}\label{Fib}
\begin{split}
& \left.\begin{array}{l}\displaystyle{\lambda_k = \mu_0 \sigma^{(F_{k+2}-2)
      + (F_{k+2}-3)N/2}} \vspace{1mm }\\ 
\displaystyle{\mu_k = \mu_0 \sigma^{(F_{k+2}-2) + (F_{k+3}-3)N/2}}
\end{array} \right\}\quad\mbox{ for all }\; k=1\ldots K-1,\\
& \; \; \lambda_K = \mu_0 \sigma^{(2F_{K}-2) + (F_{K+2}-3)N/2} ,\qquad 
\mu_K = \mu_0 \sigma^{(2F_{K}-2) + (3F_{K}-3)N/2}. 
\end{split}
\end{equation}
\end{proposition}
\begin{proof}
{\bf 1. (Conditions not involving $\gamma$)} We first directly derive the
formulas on the quotients of the frequencies $\{\lambda_i,\mu_i\}$,
which will be used below:
\begin{equation}\label{quot}
\begin{split}
& \frac{\lambda_k}{\mu_{k-1}} = \sigma^{F_k},\quad 
\frac{\mu_k}{\lambda_{k}} = \sigma^{F_{k+1}N/2} \qquad\mbox{ for all
}\;k=1\ldots K-1,\\ 
& \frac{\lambda_K}{\mu_{K-1}}= \sigma^{2F_K-F_{K+1}}=\sigma^{F_{K-2}},
\qquad \frac{\mu_K}{\lambda_{K}} = \sigma^{3F_K-F_{k+2}N/2}=\sigma^{F_{K-2}N/2}. 
\end{split}
\end{equation}
We now validate all conditions (\ref{ass_ml}), apart from the first one, as
equalities. All formulas follow by (\ref{quot}). Regarding the second condition, we have:
$$\frac{\mu_1}{\lambda_1} = \sigma^{F_2N/2} = \sigma^N =
\big(\frac{\lambda_1}{\mu_0}\big)^N.$$
The third condition holds, because for $k=2\ldots K-1$:
\begin{equation*}
\begin{split}
& \big(\frac{\lambda_k}{\mu_{k-1}}\frac{\lambda_{k-1}}{\mu_{k-2}}\big)^{N/2}
= \sigma^{(F_k+F_{k-1})N/2}= \sigma^{F_{k+1}N/2} = \frac{\mu_k}{\lambda_k}, \\
& \frac{(\lambda_k/\mu_{k-1})^N(\lambda_{k-1}/\mu_{k-2})^{N/2}}{\mu_{k-1}/\lambda_{k-1}}
=\sigma^{F_kN + F_{k-1}N/2 - F_k N/2} = \sigma^{F_{k+1}N/2}= \frac{\mu_k}{\lambda_k}.
\end{split}
\end{equation*}
Likewise, the last condition in (\ref{ass_ml}) follows from:
\begin{equation*}
\begin{split}
& \big(\frac{\lambda_K}{\mu_{K-1}}\big)^{N/2}
= \sigma^{F_{K-2}N/2} = \frac{\mu_K}{\lambda_K}, \\
& \frac{(\lambda_K/\mu_{K-1})^N(\lambda_{K-1}/\mu_{K-2})^{N/2}}{\mu_{K-1}/\lambda_{K-1}}
=\sigma^{F_{K-2} N + F_{K-1}N/2 - F_K N/2} = \sigma^{F_{K-2}N/2}= \frac{\mu_K}{\lambda_K}.
\end{split}
\end{equation*}

\smallskip

{\bf 2. (Condition involving $\gamma$)} To prove the first condition
in (\ref{ass_ml}), we use (\ref{quot}) to check that for all $k=1\ldots K-1$ there holds:
$$\frac{\lambda_k^{1-\gamma}}{\mu_{k-1}} = \frac{1}{\lambda_k^\gamma}\frac{\lambda_k}{\mu_{k-1}} 
=\frac{\sigma^{F_k}}{\lambda_k^\gamma}= \frac{\sigma}{(\mu_0\sigma)^\gamma}
\sigma^{(F_k-1) - (F_{k+2}-3)(1+N/2)\gamma}
\geq \frac{\sigma}{(\mu_0\sigma)^\gamma}\geq \sigma_0,$$
by (\ref{ass_f}) and because $(F_k-1) - (F_{k+2}-3)(1+N/2)\gamma$ is
always nonnegative. Indeed, at $k=1$ this expression is null, while
for $k>1$ it is at least $1 - (F_{k+2}-3)(1+N/2)\gamma$. The
assumption (\ref{gam_small})
guarantees that this last expression is nonnegative.
Similarly:
$$\frac{\lambda_K^{1-\gamma}}{\mu_{K-1}} = \frac{1}{\lambda_K^\gamma}\frac{\lambda_K}{\mu_{K-1}} 
=\frac{\sigma^{F_{K-2}}}{\lambda_K^\gamma}\geq \frac{\sigma}{(\mu_0\sigma)^\gamma}
\sigma^{(F_{K-2}-1) - (F_{K+2}-3)(1+N/2)\gamma}
\geq \frac{\sigma}{(\mu_0\sigma)^\gamma}\geq \sigma_0,$$
since $(F_{K-2}-1) - (F_{K+2}-3)(1+N/2)\gamma\geq 0$ when $K\geq 4$
and under (\ref{gam_small}).
\end{proof}

\bigskip

It is useful to separately derive the quantities that appear in the estimates
\ref{P3bound1} -- \ref{P3bound4}, under the definition of frequencies
in (\ref{Fib}). These are:

\begin{lemma}\label{musead}
Let $N,K\geq 4$, $\mu_0\geq 1$ and let $\{\lambda_k,\mu_k\}_{k=1}^K$
be given by (\ref{Fib}). Then there hold:
\begin{align*}
& \prod_{k=0}^{K-1}\frac{1}{(\lambda_{k+1}/\mu_k)^N} = \frac{1}{\sigma^{2(F_K-1)N}},
\vspace{3mm} \tag*{(\theequation)$_1$}\refstepcounter{equation} \label{for_D}\\
&\mu_K\Big(\prod_{k=0}^{K-2}\frac{1}{(\lambda_{k+1}/\mu_k)^{N/2}}\Big)
\Big(1+ \frac{(\lambda_{K-1}/\mu_{K-2})^{N/2}}{\mu_K/\mu_{K-1}}\Big)
\leq 2\mu_0 \sigma^{(F_{K+1}-2) + (F_{K+1}-1)N/2}.
\tag*{(\theequation)$_2$}\refstepcounter{equation} \label{for_v}
\end{align*}
\end{lemma}

\begin{proof}
Using (\ref{quot}) and the formula $\displaystyle{\sum_{i=1}^j F_i = F_{j+2} -
2}$, we obtain \ref{for_D}:
\begin{equation*}
\prod_{k=0}^{K-1}\frac{1}{(\lambda_{k+1}/\mu_k)^N} =
\sigma^{-(\sum_{k=1}^{K-1}F_k)N }\sigma^{-F_{K-2}N} = \sigma^{-(2F_K-2)N}.
\end{equation*}
Towards \ref{for_v}, note that:
\begin{equation*}
3F_{K}-3 + F_{K-1} -(F_{K+1} +F_{K-2} -2) = 2F_K -F_{K-2}-1 = F_{K+1}-1,
\end{equation*}
which implies, for $N\geq 4$:
\begin{equation*}
\begin{split}
\mu_K\Big(\prod_{k=0}^{K-2}&\frac{1}{(\lambda_{k+1}/\mu_k)^{N/2}}\Big)
\Big(1+ \frac{(\lambda_{K-1}/\mu_{K-2})^{N/2}}{\mu_K/\mu_{K-1}}\Big)
\\ & = \frac{\mu_0 \sigma^{(2F_{K}-2) + (3F_{K}-3)N/2}}{\sigma^{(\sum_{k=1}^{K-1}F_k)N/2}}
\Big(1+ \frac{\sigma^{F_{K-1}N/2}}{\sigma^{F_{K-2}(1+N/2)}} \Big)
\\ & \leq 2 \, \frac{\mu_0 \sigma^{(2F_{K}-2) + (3F_{K}-3)N/2}}{\sigma^{(F_{K+1} -2)N/2}}
\cdot \frac{\sigma^{F_{K-1}N/2}}{\sigma^{F_{K-2}(1+N/2)}} 
= 2\mu_0 \sigma^{(F_{K+1}-2) + (F_{K+1}-1)N/2},
\end{split}
\end{equation*}
since then $F_{K-2}(1+N/2)\leq 2F_{K-3} + F_{K-2} N/2\leq (F_{K-3} +
F_{K-2})N/2$. The proof is done.
\end{proof}

\bigskip

\noindent We are now ready to complete the ``stage'' construction in our
convex integration algorithm:

\bigskip

\noindent {\bf Proof of Theorem \ref{thm_stage}}.

{\bf 1. (Setting the initial quantities)}
For given $N, K\geq 4$ and $\gamma$ that satisfies
(\ref{gam_small}), we take $l_0$ as in Proposition \ref{prop3} and
$\sigma_0$ increased $(K+1)$ times. Let $v,w$ be as
in the statement of the theorem, together with the positive constants
$l,\lambda,\mathcal{M}$ satisfying (\ref{Assu}). Denote:
$$\eta=\frac{l}{K+1},\quad \mu_0=\frac{1}{\eta}.$$
We first construct the fields 
$v_0\in \mathcal{C}^\infty(\bar\omega+\bar B_{l+K\eta}(0),\R^3)$, 
$w_0\in \mathcal{C}^\infty(\bar\omega+\bar B_{l+K\eta}(0),\R^2)$, 
$A_0\in \mathcal{C}^\infty(\bar\omega+\bar B_{l+K\eta}(0),\R^{2\times
2}_\sym)$ by using the mollification kernel as in Lemma \ref{lem_stima}:
$$v_0=v\ast \phi_{\eta/2},\quad w_0=w\ast \phi_{\eta/2}, \quad A_0=A\ast \phi_{\eta/2},
\quad {\mathcal{D}}_0= A_0 - \big(\frac{1}{2}(\nabla v_0)^T\nabla v_0 + \sym\nabla w_0\big).$$
From Lemma \ref{lem_stima}, we deduce the initial bounds,
where constants $C$ depend only on $\omega,K, m$:
\begin{align*}
& \|v_0-v\|_1 + \|w_0-w\|_1 \leq C l\mathcal{M},
\tag*{(\theequation)$_1$}\refstepcounter{equation} \label{pr_stima1}\\
& \|A_0-A\|_0 \leq Cl^\beta\|A\|_{0,\beta}, \tag*{(\theequation)$_2$} \label{pr_stima2}\\
& \|\nabla^{(m+1)}v_0\|_0 + \|\nabla^{(m+1)}w_0\|_0\leq
\frac{C}{l^m} l\mathcal{M}\quad \mbox{ for all }\; m\geq 1, \tag*{(\theequation)$_3$} \label{pr_stima3}\\
& \|\nabla^{(m)} \mathcal{D}_0\|_0\leq
\frac{C}{l^m} \big(\|\mathcal{D}\|_0 + (l\mathcal{M})^2\big) \qquad \quad \mbox{ for
  all }\; m\geq 0. \tag*{(\theequation)$_4$}\label{pr_stima4} 
\end{align*}
Indeed,  \ref{pr_stima1}, \ref{pr_stima2} follow from \ref{stima2} and
in view of the lower bound
on $\mathcal{M}$. Similarly, \ref{pr_stima3} follows by applying \ref{stima1} to
$\nabla^2v$ and $\nabla^2w$ with the differentiability exponent $m-1$.
Since:
$$\mathcal{D}_0 = \mathcal{D}\ast \phi_{\eta/2} - \frac{1}{2}\big((\nabla
v_0)^T\nabla v_0 - ((\nabla v)^T\nabla v)\ast\phi_{\eta/2}\big), $$ 
we get \ref{pr_stima4} by applying \ref{stima1} to $\mathcal{D}$, and
\ref{stima4} to $\nabla v$.

\smallskip

{\bf 2. (Applying Proposition \ref{prop3})} 
Since $l+K\eta\leq 2l\leq 2l_0$, we may apply Proposition
\ref{prop3} to $v_0, w_0, \mathcal{D}_0$ with the parameters:
$$\tilde C_0=C \big(\|\mathcal{D}\|_0 + (l\mathcal{M})^2\big), \quad
\lambda_0=\mu_0=\frac{K+1}{l}, $$
consistent with \ref{pr_stima4}, and with frequencies $\{\lambda_k,
\mu_k\}_{k=1}^K$ given in Proposition \ref{prop4} for $\sigma>1$ in:
$$\sigma = \lambda l \qquad\mbox{so that: } \; 
\frac{\sigma}{(\mu_0\sigma)^\gamma} = \frac{\lambda
  l}{((K+1)\lambda)^\gamma}\geq \frac{\lambda^{1-\gamma}
  l}{K+1}\geq \sigma_0$$
In conclusion, Proposition \ref{prop3} yields $\tilde v
\in  \mathcal{C}^\infty(\bar\omega+\bar B_{l}(0),\R^3)$,  
$\tilde w\in \mathcal{C}^\infty(\bar\omega+\bar B_{l}(0),\R^2)$, 
with \ref{P3bound1} -- \ref{P3bound4}. These imply,
in virtue of Lemma \ref{musead} and the initial bounds \ref{pr_stima1} --
\ref{pr_stima4}:
\begin{equation}\label{almost}
\begin{split}
& \|\tilde v - v\|_1\leq C\Lambda^{\gamma/2}\big(\|\mathcal{D}\|_0^{1/2}
+ l\mathcal{M}\big), \\
& {\|\tilde w -w\|_1\leq C\Lambda^{\gamma}\big(\|\mathcal{D}\|_0^{1/2}
+ l\mathcal{M}\big) \big(1+ \|\mathcal{D}\|_0^{1/2} + l\mathcal{M}
+\|\nabla v\|_0\big),} \\
& \|\nabla^2\tilde v\|_0\leq C \Lambda^{\gamma/2} \frac{ (\lambda
  l)^{(F_{K+1}-2)+(F_{K+1}-1)N/2}}{l} \big(\|\mathcal{D}\|_0^{1/2} + l\mathcal{M}\big),\\
& \|\nabla^2\tilde w\|_0\leq C \Lambda^\gamma \frac{ (\lambda
    l)^{(F_{K+1}-2)+(F_{K+1}-1)N/2}}{l}
\big(\|\mathcal{D}\|_0^{1/2} + l\mathcal{M}\big)
 \big(1+\|\mathcal{D}\|_0^{1/2} + l\mathcal{M}+ \|\nabla v\|_0\big),  \\ 
& \|\tilde{\mathcal{D}}\|_0\leq C l^\beta \|A\|_{0,\beta} + C \Lambda^{\gamma/2}
\frac{1}{(\lambda
    l)^{2(F_K-1)N}}\big(\|\mathcal{D}\|_0 +(l\mathcal{M})^2\big), 
\end{split}
\end{equation}
with $C$ depending on $\omega, \gamma, N, K$. The quantity
$\Lambda$, may be estimated by:
$$\Lambda = \prod_{k=1}^K(\mu_k\lambda_k^N) \leq \mu_0^{K(N+1)}
\sigma^{P(N,K)} \leq C \frac{(\lambda l)^{P(N,K)}}{l^{K(N+1)}},$$
where $P(N,K)$ is a second order polynomial in $N$, with coefficients
depending on $K$. 

\smallskip

{\bf 3. (Reparametrizing $\gamma$)} 
Since $P(N,K)$ is (much) larger than $(2F_K-2)N$, we get:
\begin{equation*}
\begin{split}
P(N,K)-K(N+1) & \geq 2(F_K-2)N- K(N+1) 
\\ & = (F_K-K)(N+1)+(F_K-2)(N-3)+ 2(F_K-3) \geq 0,
\end{split}
\end{equation*}
because $N,K\geq 4$. Consequently:
$$\Lambda\leq C \lambda^{P(N,K)}$$
so that in (\ref{almost}) one can replace each occurrence of
$\Lambda^\gamma$ by $\lambda^{\bar\gamma}$ with $\bar\gamma = \gamma
P(N,K)\geq \gamma$. However, condition $\lambda^{1-\bar\gamma}l\geq
\sigma_0$ implies $\lambda^{1-\gamma}l\geq \sigma_0$, so the bounds (\ref{almost})
imply those in \ref{Abound12} -- \ref{Abound32}, albeit
within a smaller range of $\gamma$ than that indicated in
(\ref{gam_small}), still depending only on
$N,K$. Similarly, we finally observe that if the statement of Theorem 
\ref{thm_stage} holds for all sufficiently small $\gamma$, then it
is valid for all $\gamma\in (0,1)$, as stated.
The proof is done.
\endproof

\section{The Nash-Kuiper scheme and a proof of Theorem \ref{th_final}}\label{sec4}

The proof of Theorem \ref{th_final}  relies on iterating Theorem \ref{thm_stage} according to the
Nash-Kuiper scheme. We quote the main recursion result given in
\cite{lew_improved, lew_improved2}:, similar to 
\cite[section 6]{CDS}, but now involving the H\"older
norms, as is necessary in view of the decomposition Lemma
\ref{lem_diagonal2}.

\begin{theorem}\label{th_NK}\cite[Theorem 1.4]{lew_improved}
  \cite[Lemma 5.2]{lew_improved2}
Let $\omega\subset\R^d$ be an open, bounded and smooth domain,
and let $k, J, S\geq 1$. Assume that there exists $l_0\in (0,1)$ such that
the following holds for every $l\in (0, l_0]$. Given
$v\in\mathcal{C}^2(\bar\omega+\bar B_{2l}(0), \R^k)$,
$w\in\mathcal{C}^2(\bar\omega+\bar B_{2l}(0), \R^d)$,  
$A\in\mathcal{C}^{0,\beta}(\bar\omega+\bar B_{2l}(0), \R^{d\times
  d}_\sym)$, and $\gamma, \lambda, \mathcal{M}$ which satisfy,
together with $\sigma_0\geq 1$ that depends on $\omega, k, S, J,\gamma$: 
\begin{equation}\label{ass_impro2}
\gamma\in (0,1),\qquad \lambda^{1-\gamma} l>\sigma_0,
\qquad \mathcal{M}\geq \max\{\|v\|_2, \|w\|_2, 1\},
\end{equation}
there exist  $\tilde v\in\mathcal{C}^2(\bar \omega+\bar B_l(0),\R^k)$,
$\tilde w\in\mathcal{C}^2(\bar\omega+\bar B_l(0),\R^d)$ satisfying:
\begin{align*}
& \hspace{-3mm} \left. \begin{array}{l} \|\tilde v - v\|_1\leq
C\lambda^{\gamma/2}\big(\|\mathcal{D}\|_0^{1/2}+l\mathcal{M}\big), \vspace{1mm} \\ 
\|\tilde w - w\|_1\leq C\lambda^{\gamma}\big(\|\mathcal{D}\|_0^{1/2}+l\mathcal{M}\big)
\big(1+ \|\mathcal{D}\|_0^{1/2}+l\mathcal{M}+\|\nabla v\|_0\big), \end{array}\right.
\vspace{5mm}\\
& \hspace{-3mm} \left. \begin{array}{l} \|\nabla^2\tilde v\|_0\leq C{\displaystyle{\frac{(\lambda
 l)^J}{l}\lambda^{\gamma/2}}}\big(\|\mathcal{D}\|_0^{1/2}+l\mathcal{M}\big),\vspace{1mm}\\ 
\|\nabla^2\tilde w\|_0\leq C{\displaystyle{\frac{(\lambda
  l)^J}{l}}}\lambda^{\gamma}\big(\|\mathcal{D}\|_0^{1/2}+l\mathcal{M}\big)
\big(1+\|\mathcal{D}\|_0^{1/2}+l\mathcal{M}+\|\nabla v\|_0\big), \end{array}\right.
\medskip\\ 
& \|\tilde{\mathcal{D}}\|_0\leq C\Big(l^\beta{\|A\|_{0,\beta}}
+\frac{\lambda^{\gamma}}{(\lambda l)^S} \big(
\|\mathcal{D}\|_0 + (l\mathcal{M})^2\big)\Big).
\end{align*}
with constants $C$ depending only on $\omega, k, J,S,\gamma$,
and with the defects, as usual, denoted by: 
$$\mathcal{D}=A -\big(\frac{1}{2}(\nabla v)^T\nabla v + \sym\nabla
w\big),\qquad \tilde{\mathcal{D}}=A -\big(\frac{1}{2}(\nabla \tilde
v)^T\nabla \tilde v + \sym\nabla \tilde w\big).$$
Then, for every triple of fields $v, w, A$ as above, which additionally
satisfy the defect smallness condition $0<\|\mathcal{D}\|_0\leq 1$, 
and for every exponent $\alpha$ in the range:
\begin{equation}\label{rangeAlz}
0< \alpha <\min\Big\{\frac{\beta}{2},\frac{S}{S+2J}\Big\},
\end{equation}
there exist $\bar v\in\mathcal{C}^{1,\alpha}(\bar\omega,\R^k)$ and
$\bar w\in\mathcal{C}^{1,\alpha}(\bar\omega,\R^d)$ with the following properties:
\begin{align*}
& \|\bar v - v\|_1\leq C \big(1+\|\nabla v\|_0\big)^2
\|\mathcal{D}_0\|_0^{1/4}, \quad \|\bar w -
w\|_1\leq C(1+\|\nabla v\|_0)^3\|\mathcal{D}\|_0^{1/4}, \vspace{1mm}\\
& A-\big(\frac{1}{2}(\nabla \bar v)^T\nabla \bar v + \sym\nabla
\bar w\big) =0 \quad\mbox{ in }\; \bar\omega. 
\end{align*}
The constants $C$ above depend only on $\omega, k, A$ and $\alpha$. 
\end{theorem}

\smallskip

Clearly, Theorem \ref{th_NK} and Theorem \ref{thm_stage} yield together
the following result below, where we compute $\frac{S}{S+2J} =
\frac{1}{1+2J/S}$, with $\varphi$ denoting the golden ratio in: 
\begin{equation*}
\begin{split}
& \frac{J}{S}=\frac{(F_{K+1}-2) + (F_{K+1}-1)N/2}{2N(F_K-1)} \to
\frac{F_{K-1}-1}{4(F_K-1)} \quad \mbox{ as } K\to\infty \\
& \mbox{and: } \; \frac{F_{K-1}-1}{4(F_K-1)} \to \frac{\varphi}{4} = \frac{1+\sqrt{5}}{8} \quad \mbox{ as } N\to\infty.
\end{split}
\end{equation*}

\begin{corollary}\label{th_NKH}
Let $\omega\subset\R^2$ be an open, bounded and smooth domain.
Fix any $\alpha$ as in (\ref{VKrange}).
Then, there exists $l_0\in (0,1)$ such that, for every $l\in (0, l_0]$,
and for every $v\in\mathcal{C}^2(\bar\omega + \bar B_{2l}(0),\R^3)$,
$w\in\mathcal{C}^2(\bar\omega +\bar B_{2l}(0),\R^2)$,
$A\in\mathcal{C}^{0,\beta}(\bar\omega +\bar B_{2l}(0), \R^{2\times 2}_\sym)$ such that:
$$\mathcal{D}=A-\big(\frac{1}{2}(\nabla v)^T\nabla v + \sym\nabla
w\big) \quad\mbox{ satisfies } \quad 0<\|\mathcal{D}\|_0\leq 1,$$
there exist $\tilde v\in\mathcal{C}^{1,\alpha}(\bar\omega,\R^3)$,
$\tilde w\in\mathcal{C}^{1,\alpha}(\bar\omega,\R^2)$ with the following properties:
\begin{align*}
& \|\tilde v - v\|_1\leq C(1+\|\nabla v\|_0)^2\|\mathcal{D}\|_0^{1/4}, \quad \|\tilde w -
w\|_1\leq C (1+\|\nabla v\|_0)^3\|\mathcal{D}\|_0^{1/4}, \vspace{1mm}\\
& A-\big(\frac{1}{2}(\nabla \tilde v)^T\nabla \tilde v + \sym\nabla
\tilde w\big) =0 \quad\mbox{ in }\; \bar\omega. 
\end{align*}
The norms in
the left hand side above are taken on $\bar\omega$, and in the right hand
side on $\bar\omega+ \bar B_{2l}(0)$. The constants $C$ depend only
on $\omega, A$ and $\alpha$. 
\end{corollary}

\bigskip

\noindent The proof of Theorem \ref{th_final} is consequently the same as the proof of
Theorem 1.1 in \cite{lew_improved}, in section 5 in there. We 
replace $\omega$ by its smooth superset, and apply the basic stage
construction in order to first decrease $\|\mathcal{D}\|_0$
below $1$. Then, Corollary \ref{th_NKH} yields the theorem. \endproof

\end{document}